\documentclass[a4paper, pdftex]{amsart}

\usepackage[colorlinks, linkcolor=blue, citecolor=blue]{hyperref}
\usepackage[utf8]{inputenc}
\usepackage[T1]{fontenc}
\usepackage[pdftex]{graphicx}
\usepackage{amsmath}
\usepackage{amsfonts}
\usepackage{amsthm}
\usepackage{amssymb}
\usepackage{thmtools}
\usepackage{thm-restate}
\usepackage{mathtools}
\usepackage{hyperref}
\usepackage[capitalize,nameinlink,noabbrev]{cleveref}
\usepackage{verbatim}
\usepackage[all, cmtip]{xy}
\usepackage{multirow}
\usepackage{bm}
\usepackage[textsize = tiny]{todonotes}
\usepackage{makecell}
\usepackage{enumitem}
\usepackage[margin=1in]{geometry}

\usepackage{tikz-cd}

\usepackage[backend=biber, style=alphabetic, url=true, hyperref, maxbibnames=99, sorting=nyt, eprint=false, isbn=false, sortcites=true]{biblatex}
\AtEveryBibitem{\clearlist{language}}
\newtheorem*{thm*}{Theorem}
\newtheorem{theorem}{Theorem}[section]
\newtheorem{corollary}[theorem]{Corollary}
\newtheorem{lemma}[theorem]{Lemma}

\newtheorem*{fact*}{Fact}
\newtheorem{proposition}[theorem]{Proposition}

\newcounter{theoremalph}

\theoremstyle{definition}
\newtheorem{question}[theorem]{Question}
\newtheorem{definition}[theorem]{Definition}

\theoremstyle{remark}
\newtheorem{remark}[theorem]{Remark}

\newtheorem{example}[theorem]{Example}

\newcommand{\ls}{\left\{}
\newcommand{\rs}{\right\}}
\newcommand{\sm}{-}

\newcommand{\Z}{\ensuremath{\mathbb{Z}}}

\newcommand{\mcB}{\ensuremath{\mathcal{B}}}

\newcommand{\mbN}{\ensuremath{\mathbb{N}}}

\newcommand{\Star}{\ensuremath{\st}}
\newcommand{\Link}{\ensuremath{\operatorname{lk}}}
\newcommand{\Cone}{\ensuremath{\operatorname{Cone}}}

 \newcommand\gen[1]{\left\langle#1\right\rangle}
 \newcommand\Span[1]{\left\langle#1\right\rangle}

\newcommand{\Aut}{\ensuremath{\operatorname{Aut}}}
\newcommand{\Diff}{\ensuremath{\operatorname{Diff}}}

\newcommand{\Out}{\ensuremath{\operatorname{Out}}}

\newcommand{\St}{\operatorname{St}}
\newcommand{\gtl}{\mathfrak{grt}}

\newcommand{\op}{\operatorname{op}}
\newcommand{\st}{\operatorname{st}}
\newcommand{\sd}{\operatorname{sd}}
\newcommand{\lk}{\operatorname{lk}}

\newcommand{\im}{\operatorname{im}}
\newcommand{\rk}{\operatorname{rk}}

\newcommand{\GL}[2]{\ensuremath{\operatorname{GL}_{#1}(#2)}}
\newcommand{\SL}[2]{\ensuremath{\operatorname{SL}_{#1}(#2)}}

\newcommand{\Decomp}{\operatorname{D}}

\newcommand{\PD}{\operatorname{PD}}
\newcommand{\CB}{\operatorname{CB}}
\newcommand{\FC}{\operatorname{FC}}

\newcommand{\Kfour}{\operatorname{FCD}}

\newcommand{\factorposet}{\mathcal{P}}

\newcommand{\spine}{K_n}

\newcommand{\geomPD}{\mathrm{FCD}^{\mathrm{g}}}
\newcommand{\geomD}{\mathrm{D}^{\mathrm{g}}}
\newcommand{\sphere}{\mathcal{S}}
\newcommand{\boundcomp}{M_0}
\newcommand{\decompMap}{\mathbf{d}}
\newcommand{\MCG}{\operatorname{MCG}}

\newcommand{\thedegree}{l} 
\newcommand{\themaxdegree}{L}
\newcommand{\Ftau}{F_{\mathrm{cut}}} 
\newcommand{\Fsigma}{\sphere_{A}} 
\newcommand{\Fdecomp}{P_1} 
\newcommand{\Fpartdecomp}{P_2} 

\newcommand{\tq}{\mathrel{{\ensuremath{\: : \: }}}}

\newcommand{\m}{\to}

\title{An analogue of Rognes' connectivity conjecture for free groups}

\author{Benjamin Br{\"u}ck}
\address{Universit{\"a}t M{\"u}nster \\
Institut f{\"u}r Mathematische Logik und Grundlagenforschung \\
48149 M{\"u}nster, Germany}
\email{benjamin.brueck@uni-muenster.de}

\author{Jeremy Miller}
\address{Purdue University \\
Department of Mathematics\\
47907 West Lafayette, United States}
\email{jeremykmiller@purdue.edu}

\author{Kevin I. Piterman}
\address{Vrije Universiteit Brussel \\
Department of Mathematics and Data Science \\
1050 Brussels, Belgium}
\email{kevin.piterman@vub.be}

\keywords{Free groups; Sphere systems; Common basis complex}

\subjclass[2020]{20F65, 20E05, 55P48, 57M07}

\begin{document}

\begin{abstract}
We show that the common basis complex of a free group of rank $n$ has the homotopy type of a wedge of spheres of dimension $2n-3$. This establishes an $\Aut(F_n)$-analogue of the connectivity conjecture that Rognes originally stated for $\GL{n}{R}$.
To prove this, we provide several homotopy-equivalent models of the common basis complex, both in terms of free factors in free groups and in terms of sphere systems in 3-manifolds. 
\end{abstract}

\maketitle

\vspace{-.3in}

\section{Introduction}
\subsection{Rognes' connectivity conjecture and the main theorem}

For $R$ a sufficiently nice ring (e.g. commutative), Rognes \cite{ROGNES1992813} defined a rank filtration of the free algebraic $K$-theory spectrum of $R$. He showed that the associated graded is given by the general linear group homotopy orbits of the suspension spectrum of $\Sigma \CB(R^n)$, with $\CB(R^n)$ the common basis complex of $R$. This simplicial complex has vertices the proper nonzero direct summands of $R^n$ and a collection of such summands forms a simplex if there is a \emph{common basis}, i.e., a basis of $R^n$ such that each summand is spanned by some subset of the basis. Motivated by potential applications to computing $K_*(\Z)$, Rognes conjectured that $\CB(R^n) \simeq \bigvee S^{2n-3}$ if $R$ is Euclidean or local. This conjecture is known to be true for $R$ a field by work of the second author with Patzt and Wilson \cite{miller2023}; for finite fields, it also follows by combining work of the first author, third author, and Welker \cite{Brueck2024} with work of Hanlon--Hersh--Shareshian \cite{HHS}. The general case is still open, even for $R=\Z$.

In this paper, we prove an analogue of Rognes' connectivity conjecture for free groups. 
That is, we consider a version of the complex $\CB(R^n)$ for free groups and show that it is highly connected: A subgroup $H \subseteq F_n$ of the free group $F_n$ of rank $n$ is a \emph{free factor} if there is a subgroup $K \subseteq F_n$ with internal free product $H * K =F_n$. While already $F_2$ contains free subgroups of arbitrary rank, all free factors of $F_n$ have rank at most $n$. In this sense, free factors of $F_n$ can be seen as analogues of direct summands of $R^n$. The common basis complex for $F_n$, denoted by $\CB(F_n)$, is the simplicial complex whose vertices are non-trivial proper free factors of $F_n$ and a collection of such factors forms a simplex if there is a common basis in the sense of free groups (see \cref{sec:algebraic_complexes}). Our main theorem is the following.

\begin{theorem} \label{thm:mainCB}
For all $n \geq 1$, $\CB(F_n) \simeq \bigvee S^{2n-3}$.
\end{theorem}

The most direct outcome of \cref{thm:mainCB} and its proof is that we provide several highly-connected $\Aut(F_n)$-complexes (see \cref{algebraicModels} and \cref{sec:intro_geometry} for alternative descriptions of $\CB(F_n)$). High connectivity results for simplicial complexes with group actions are useful for studying homological and geometric properties of groups. In particular, families of highly-connected simplicial complexes are important in the study of homology stability. There are a plethora of families of simplicial complexes with slope-$1$ connectivity ranges (see Randal-Williams--Wahl \cite{RWW} and the references therein). In contrast, simplicial complexes with slope-$2$ connectivity ranges like the one we obtain in \cref{thm:mainCB} are much rarer (see e.g.~\cite{HarerVCD,HHS,GKRW4,miller2023}). These better connectivity ranges sometimes imply patterns for homology outside of the stable range (see Galatius--Kupers--Randal-Williams \cite[Theorem E]{GKRW4}).

Maybe even more importantly, our results give new insights into two related open questions. The first is whether $\CB(\Z^n)$ is homotopy equivalent to a wedge of $(2n-3)$-spheres, and the second is what the homotopy type of the simplicial boundary of Culler--Vogtmann's Outer space is.

The first of these questions is Rognes' connectivity conjecture for $R=\Z$.
It has deep connections to the cohomology of arithmetic groups and algebraic $K$-theory. For example, it implies Church--Farb--Putman's vanishing conjecture for $\SL{n}{\Z}$ \cite[Conjecture 2]{CFPconj}; see Charlton--Radchenko--Rudenko \cite[Section 1.3]{SteinbergPoly}. It also implies that $K_{12}(\Z) \cong 0$; see Dutour--Sikiri\'c--Elbaz-Vincent--Kupers--Martinet \cite[Remark 6.6]{K8}.
Solving this conjecture seems to be a very hard problem. Previous attempts focused on algebraic techniques \cite{ROGNES1992813,GKRW4,HHS,miller2023}. In contrast to that, our methods suggest that a more ``geometric'' approach that studies related simplicial complexes could be more tractable.
A different approach using the theory of buildings was taken in \cite{PitermanShareshianWelker}.

The second question about the homotopy type of the simplicial boundary of Outer space was raised by Br\"uck--Gupta \cite[Question 1.1]{BG:Homotopytypecomplex} and by Vogtmann \cite[second question in Section 3.2]{Vogtmann2024}. It is motivated by questions related to the dualizing module of $\Out(F_n)$, especially in comparison to the related settings of arithmetic groups and mapping class groups \cite{BF:topologyinfinity,BSV:bordificationouterspace,Vogtmann2024, Wade2024a}.
To prove \cref{thm:mainCB}, we show that certain subcomplexes of Auter space, the moduli space of metric graphs with a choice of basepoint \cite[Section 2]{HatcherVogtmannCerf}, are highly connected (see \cref{sec:intro_geometry}). These results give an indication as to what the homotopy type of their $\Out(F_n)$-version could be.

In the following three subsections, we describe our approach for proving \cref{thm:mainCB}. We return to these analogies and open questions in \cref{sec:analogies}.

\subsection{Alternative algebraic models of \texorpdfstring{$\CB(F_n)$}{CB(Fn)}} \label{algebraicModels}

The simplicial complex $\CB(F_n)$ has dimension $2^n -3$. Nevertheless, \cref{thm:mainCB} implies that $\CB(F_n)$ is homotopy equivalent to a $(2n-3)$-dimensional $CW$-complex. In fact, this was previously known by work of the first author, third author, and Welker \cite{Brueck2024}. They showed that the common basis complexes of both rings and free groups have a description in terms of the \emph{poset of partial decompositions}. A non-empty collection of proper non-trivial free factors $\{N_1,\ldots,N_k\}$ in $F_n$ is a partial decomposition if $F_n = N_1\ast \cdots \ast N_k \ast K$ for some $K$. Refinement gives a natural partial ordering on the set of partial decompositions, see \cref{sec:algebraic_complexes}. We denote this poset by $\PD(F_n)$. It has dimension $2n-3$ and \cite[Corollary 1.1]{Brueck2024} implies that $\PD(F_n)\simeq \CB(F_n)$, so \cref{thm:mainCB} is equivalent to showing $\PD(F_n)$ is $(2n-4)$-connected.

We will prove that $\PD(F_n)$ is $(2n-4)$-connected by first comparing it to another complex of free factors. Let $\Decomp(F_n) \subset \PD(F_n)$ be the subposet of \emph{full decompositions}, the partial decompositions $\{N_1,\ldots,N_k\}$ that generate $F_n = N_1\ast \cdots \ast N_k$.
Let $\FC_n$ be the poset of non-trivial proper free factors of $F_n$, with order given by inclusion.
The order complex of $\FC_n$ is the \emph{free factor complex} of Hatcher--Vogtmann \cite{HV2022freefactors}, who also denote it by $\FC_n$.
We combine $\Decomp(F_n)$ and  $\FC_n$ to form a new poset called $\Kfour_n$ which can be seen as the subcomplex of the join $\Decomp(F_n)\ast\FC_n$ consisting of ``basis compatible pairs'' (for the precise definition, see \cref{eq:defKfour}).  A key step in our proof is to show that $\Kfour_n$ is homotopy equivalent to $\PD(F_n)$ and hence to $\CB(F_n)$.

\subsection{Geometric models of \texorpdfstring{$\CB(F_n)$}{CB(Fn)} and sphere systems}
\label{sec:intro_geometry}

Our proof that $\CB(F_n)$ is $(2n-4)$-connected is inspired by a proof of Hatcher--Vogtmann that shows that $\Decomp(F_n)$ is $(n-3)$-connected \cite[Theorem 6.1]{HatcherVogtmannCerf}. To obtain their result, they first give a ``geometric'' model $\geomD_n\simeq \Decomp(F_n)$ defined in terms of sphere systems in $3$-manifolds. 

Hatcher \cite[Appendix]{Hat:Homologicalstabilityautomorphism} described how to view Auter space and its simplicial closure as spaces of $2$-spheres in $3$-manifolds. Let $M_{n,m}$ be the connected sum of $n$ copies of $S^1 \times S^2$ with $m$ open balls removed. 
Let $\sphere(M_{n,m})$ denote the simplicial complex whose simplices are  \emph{sphere systems} in $M_{n,m}$, i.e.,~sets $\{[S_1],\ldots,[S_r]\}$ of isotopy classes of disjointly embedded $2$-spheres $S_i\hookrightarrow M_{n,m}$ that are neither isotopic to a boundary component nor bound a disk.
Hatcher proved that $\sphere(M_{n,0})$ is the simplicial closure of Outer space and $\sphere(M_{n,1})$ is the simplicial closure of Auter space. 
The \emph{spine} of Auter space is the subposet $\spine \subset \sphere(M_{n,1})$ of sphere systems $\ls [S_1], \ldots, [S_r]\rs$ such that all path components of $M_{n,1} \setminus(S_1 \cup \ldots \cup  S_r)$ are simply connected.\footnote{Note that $K_n$ is not a subcomplex of $\sphere(M_{n,1})$ (the face of a simplex in $K_n$ need not be in $K_n$ anymore), which is why we only talk about subposets here. However, $K_n$ can be interpreted as a subcomplex of the barycentric subdivision of $\sphere(M_{n,1})$.}
Every sphere system $\sigma=\{[S_1],\ldots,[S_r]\}$ has an associated \emph{dual graph} $\Gamma(\sigma)$. The vertices of $\Gamma(\sigma)$ are the path components of $M_{n,m} \setminus (S_1 \cup \ldots \cup  S_r)$. There is an edge for each $S_i$ connected to the (possibly not distinct) vertices corresponding to the two components on either side of the sphere. For sphere systems in $M_{n,1}$, there is a unique connected component $\boundcomp(\sigma)$ of $M_{n,1} \setminus (S_1 \cup \ldots \cup  S_r)$ that contains the boundary of $M_{n,1}$. This component corresponds to a vertex of $\Gamma(\sigma)$.

Hatcher--Vogtmann defined $\geomD_n\subset \spine$ as the subposet of sphere systems $\sigma$ where $|\Gamma(\sigma)| \setminus \{ \boundcomp(\sigma) \}$ is disconnected ($\boundcomp(\sigma)$ is a \emph{cut vertex}). A key step in their proof that $\Decomp(F_n)$ is spherical is to show that $\geomD_n \simeq \Decomp(F_n)$.

Our sphere-system model of $\Kfour_n \simeq\CB(F_n)$ is a subposet $\geomPD_n$ of $\sphere(M_{n,1})$ such that its intersection with the spine $\spine$ is equal to Hatcher--Vogtmann's $\geomD_n$. Namely, we let $\geomPD_n \subset \sphere(M_{n,1})$ denote the subposet of sphere systems $\sigma$ such that either $\boundcomp(\sigma)$ is a cut vertex of $\Gamma(\sigma)$ or $\boundcomp(\sigma)$ has non-trivial fundamental group. We prove the following.

\begin{theorem}
\label{DsimCB}
There is a homotopy equivalence $\geomPD_n \simeq \CB(F_n)$.
\end{theorem}

\cref{DsimCB} is the most technically involved part of the paper. We prove \cref{DsimCB} by studying various maps of posets and using contractibility of Auter space and its variants. Although one can construct maps from $\CB(F_n)$ and $\PD(F_n)$ to $\geomPD_n$ directly, it is technically much easier to compare $\geomPD_n$ to $\Kfour_n$.

The following diagram summarizes the homotopy equivalences between various models of $\CB(F_n)$ that we will need:
\begin{equation}
\label{eq:he_CB_PD_geomPD}
\begin{tikzcd}[row sep=large, column sep=large]
  & \PD(F_n) \arrow[dl, "\psi"', "\mathrm{Thm} \, \ref{thm:CBandPD}"] \arrow[dr, "\mathrm{Thm} \, \ref{thm:he_kfour}", "\varphi"'] & & \geomPD_n \arrow[dl, "\alpha"', "\mathrm{Thm} \, \ref{thm:he_geomPD_kfour}"]\\
  \CB(F_n) & & \Kfour_n
\end{tikzcd}    
\end{equation}
In fact, all the maps in \cref{eq:he_CB_PD_geomPD} are $\Aut(F_n)$-equivariant.

\subsection{Degree theorem} \label{proofStrategy} 
To show that $\geomD_n$ is spherical, Hatcher--Vogtmann filter the spine $\spine$ by subspaces of sphere systems with ``degree'' (see \cref{def:degree}) bounded by a fixed number. They prove a so-called degree theorem \cite[Corollary 3.2]{HatcherVogtmannCerf} which states that the subcomplexes in this filtration are highly connected in a range increasing as the degree tends to infinity.
They use this to prove that $\geomD_n$ is highly connected.

Aygun and the second author \cite[Theorem 1.2]{Aygun2022} proved a degree theorem for the simplicial closure of Auter space, i.e.,~for $\sphere(M_{n,1})$. In the present work, we use this result to prove that $\geomPD_n$ is $(2n-4)$-connected (see \cref{prop:connectivity_D}) which then implies \cref{thm:mainCB} using \cref{eq:he_CB_PD_geomPD}.

\subsection{Analogies and questions}
\label{sec:analogies}

\subsubsection{Rank filtrations and \texorpdfstring{$\CB(F_n)$}{CB(Fn)}} 

Rognes' interest in the connectivity of the common basis complex stemmed from its appearance in rank filtrations of algebraic $K$-theory.  
By work of Galatius \cite{SorenAutFn}, one has $$\Omega B \left(\bigoplus_n B\Aut(F_n) \right) \simeq \Omega^\infty \mathbb{S}.$$ Thus, the sphere spectrum $\mathbb{S}$ can be thought of as the algebraic $K$-theory spectrum of free groups. 
\begin{question} Does
$\mathbb{S}$ have a natural ``rank filtration'' whose associated graded is given by the $\Aut(F_n)$ homotopy orbits of $\Sigma^{\infty +1} \CB(F_n)$?
\end{question}

Compare with work of Brown--Chan--Galatius--Payne \cite{BCGP} on graph spectral sequences.

\subsubsection{Derived indecomposables}

It is interesting to compare our proof with that of the second author with Patzt and Wilson \cite{miller2023} (which closely follows a strategy developed by Galatius--Kupers--Randal-Williams \cite{GKRW4} to prove a version of Rognes' connectivity conjecture after taking homotopy orbits). To prove that $\CB(R^n)$ is highly connected for $R$ a field, they first show that $\CB(R^n)$ measures the commutative homology (also known as Harrison homology) of an equivariant ring made out of Steinberg modules. The Steinberg modules $\St(R^n)$ are defined as $\widetilde H_{n-2}(T(R^n))$ with $T(R^n)$ the Tits building associated to $\SL{n}{R}$ (the geometric realization of the poset of proper nonzero summands of $R^n$). Then they prove a vanishing line for the associative homology of the Steinberg ring by comparing it to the join $T(R^n) *T(R^n)$. Since a vanishing line for associative homology gives a vanishing line for commutative homology, it follows that $\widetilde H_*(\CB(R^n))$ vanishes in a range.

The replacement for the Tits building in this context is Hatcher--Vogtmann's free factor complex $\FC_n$. Call the top homology group of the free factor complex $\St(F_n)$ (although these representations are not generally the dualizing modules of $\Aut(F_n)$ \cite{10.1093/imrn/rnac330}). 

\begin{question}
    Does the reduced homology of $\CB(F_n)$ measure the commutative homology of an equivariant ring built out of the groups $\St(F_n)$? 
\end{question}

A possible approach involves comparing \cref{thm:he_kfour} with Campbell--Kupers--Zakharevich \cite[Section A.2]{KupAppendix}. 

\subsubsection{Resolutions}

The degree theorem for the simplicial closure of Auter space gives partial resolutions of $\St(F_n)$ in the same way that the highly-connected simplicial complexes of \cite{CFP,CP,BMPSW} give partial resolutions of $\St(\Z^n)$. Our proof that $\CB(F_n)$ is $(2n-4)$-connected should be compared with \cite[Theorem C]{MPPreductive} which implies $\CB(\Z^n)$ is $n$-connected for $n \geq 4$. From these partial resolutions of $\St(\Z^n)$, one can read off a vanishing line for the commutative homology of $\St(\Z)$ which can be used to show $\CB(\Z^n)$ is highly connected. Since these partial resolutions of $\St(\Z)$ only currently compute the $k$-syzygies for $k \leq 2$, the full connectivity conjecture remains open for $\CB(\Z^n)$. In contrast, in the $\Aut(F_n)$-case, these partial resolutions provide information about the $k$-syzygies for all $k$ and so we are able to establish a larger connectivity range.


One of the reasons that  $\CB(F_n)$ seems to be more tractable than $\CB(\Z^n)$ is that Voronoi's perfect cone tesselation for a bordification of the symmetric space associated to $\SL{n}{\Z}$ is more complicated than the simplicial complex structure on the simplicial closure of Auter space. If one could establish a ``degree theorem'' for the Voronoi tesselation, Rognes' connectivity conjecture for $\CB(\Z^n)$ would likely follow.

\subsubsection{Graph homology and \texorpdfstring{$\Out(F_n)$}{Out(Fn)}}

There is a natural variant of $\PD(F_n)$ adapted to the group $\Out(F_n)$ instead of $\Aut(F_n)$. Let $\PD^{conj}(F_n)$ denote the poset of collections of $F_n$-conjugacy classes $\{[N_1],\ldots, [N_k]\}$, where $\{N_1,\ldots, N_k\}$ is a partial decomposition of $F_n$. This complex was first defined by Handel--Mosher, who denoted it by $\mathcal{FF}(F_n,\emptyset)$ in \cite{HM:Relativefreesplitting}. The first author and Gupta \cite[Theorem B]{BG:Homotopytypecomplex} proved that $\PD^{conj}(F_n)$ is homotopy equivalent to the boundary of Outer space. They used this to prove that 
$\PD^{conj}(F_n)$ is $(n-2)$-connected and asked if in fact it was more highly connected \cite[Question 1.1]{BG:Homotopytypecomplex}. Our result that $$\PD(F_n) \simeq \bigvee S^{2n-3}$$ gives evidence for the belief that $\PD^{conj}(F_n)$ should in fact be $(2n-4)$-connected. Further indications for this are given in the work of the first and third authors \cite{Brueck2024b}. If that was the case, this would give a negative answer to questions by the first author \cite[Paragraph after Question 4.50]{Bru:buildingsfreefactora} and Vogtmann \cite[second question in Section 3.2]{Vogtmann2024}. If one were to use the techniques of the paper to study $\PD^{conj}(F_n)$, the definition of degree will need to be significantly modified since the version we consider fundamentally uses that the dual graphs have a basepoint.

Combining work of Conant--Vogtmann \cite[Proposition 27]{ConantVogtmann} with work of the first author and Gupta \cite[Theorem B]{BG:Homotopytypecomplex}, we obtain an isomorphism
$$H_i(\Out(F_n); \widetilde C_*(\PD^{conj}(F_n))  ) \cong H_{i-2n-2}(GC_2^{(n)})$$ 
between these equivariant hyper homology groups and the $n$th graded part of Kontsevich's commutative graph homology\footnote{We here follow the convention that a graph lives in $GC_2^{(n)}$ if its fundamental group has rank $n$ and then it has homological degree $k-(n+1)$, where $k$ is the number of its vertices. This agrees with the convention of Willwacher \cite{Willwacher} but is different from the one that Kontsevich originally used.} \cite{Kontsevich}. Combining this with the work of Willwacher \cite[Theorem 1.1]{Willwacher}, we see that 
$$H^{2n-2}\big(\Out(F_n);  \widetilde  C_*(\PD^{conj}(F_n))\big) \cong H^0(GC_2^{(n)}) \cong (\gtl_1)_n$$
is the $n$th graded piece of the Grothendieck-Teichm\"uller Lie algebra $\gtl_1$. If $\PD^{conj}(F_n)$ were shown to be $(2n-4)$-connected, it would imply the existence of an isomorphism
$$\bigoplus_n H^{1}\big(\Out(F_n);  \widetilde H_{2n-3}(\PD^{conj}(F_n)) \big) \cong \gtl_1 .$$
Since $\PD(F_n)$ is spherical, $$H^{2n-2}\big(\Aut(F_n); \widetilde  C_*(\PD(F_n)) \big) \cong H^1\big(\Aut(F_n);H_{2n-3}(\PD(F_n) ) \big).$$
Let
$$\mathcal L:=\bigoplus_n H^1\big(\Aut(F_n);H_{2n-3}(\PD(F_n) )\big).$$ 
The map $\Aut(F_n) \to \Out(F_n)$ gives a map \begin{equation} \label{gtl}
\gtl_1 \to \mathcal L.\end{equation}

\begin{question}
    Is there a natural Lie algebra structure on $\mathcal L$ making the map from \cref{gtl} a map of Lie algebras? 
\end{question}


If the sequence of representations $\St(F_n)$ has a duoidal bi-algebra structure in the sense of \cite[Remark 1.4]{AMP}, then 
$\mathcal L$ would naturally have the structure of a Lie algebra (also see Brown--Chan--Galatius--Payne \cite{BCGP}).

\subsection{Outline}

In \cref{Setup}, we review background on posets and simplicial complexes in general as well as those relevant to the paper. In \cref{FCcbD}, we prove that $\CB(F_n) \simeq \Kfour_n$. Then, in \cref{geoComplex}, we describe a map $\geomPD_n \to \Kfour_n$ and prove that it is a homotopy equivalence in \cref{sec:alpha_equiv}. Finally, in \cref{sec:connectivity_D}, we prove that the sphere system model $\geomPD_n$ is highly connected.

\subsection{Acknowledgments}
We thank José Joaquín Domínguez Sánchez, Radhika Gupta, Alexander Kupers, Peter Patzt, John Rognes, Robin J.~Sroka, and Jennifer Wilson for helpful conversations. 

The first author was supported by the Deutsche Forschungsgemeinschaft through Germany’s Excellence Strategy grant EXC 2044/2–390685587 and through Project 427320536–SFB 1442. The second author was supported by a Simons Foundation Travel Support for Mathematicians grant and NSF grants DMS-2202943 and DMS-2504473. The third author was supported by the FWO grant 12K1223N.

\section{Setup, preliminaries} \label{Setup}

In this section, we start by reviewing generalities about posets and simplicial complexes. Then we describe some posets of free factors. We end with a discussion of sphere systems in certain $3$-manifolds. All of the complexes and posets described in the present section have already appeared in previous work.

\subsection{Notation and poset terminology}

Let $X$ be a poset.
We write $\sd X$ for the poset of finite non-empty chains of $X$.
Sometimes, to relax the notation, we regard a simplicial complex as a poset via its face poset of non-empty simplices.
Likewise, we view a poset as a topological space, equipped with the topology induced by its order complex.
If $f,g:X\to Y$ are order-preserving maps between posets and $f(x)\leq g(x)$ for all $x\in X$, then $f,g$ give rise to homotopic maps between the geometric realizations.

Let $Y$ be a subposet of $X$.
For $y\in X$, the lower interval $Y_{\leq y}$ consists of the elements $x\in Y$ such that $x\leq y$.
We say that $Y$ is downward closed if $Y_{\leq y} = X_{\leq y}$ for all $y\in Y$.
We can analogously define $Y_{<y}$, $Y_{\geq y}$, $Y_{>y}$ and the upward-closed property.

Our main tool to prove that poset maps are homotopy equivalences is the well-known Quillen  Fiber Theorem (or Theorem A, see \cite{Qui78}):

\begin{theorem}
[Quillen]
\label{thm:quillen}
Let $f:X\to Y$ be an order-preserving map between posets.
Assume that for all $y\in Y$, the lower fiber $f^{-1}(Y_{\leq y})$ is contractible.
Then $f$ is a homotopy equivalence.
\end{theorem}

\subsection{Algebraic complexes}
\label{sec:algebraic_complexes}

Let $\factorposet$ be the set of all non-empty sets of non-trivial proper free factors of $F_n$.
There are two natural poset relations on $\factorposet$: One given by \emph{inclusion}, $c\leq d$ if $c\subseteq d$, and one given by \emph{refinement}, $c\leq d$ if for all $A\in c$, there is $B\in d$ with $A\subseteq B$.
We consider three posets on elements of $\factorposet$:
\begin{itemize}
\item $\FC_n$ is the poset of non-trivial proper free factors of $F_n$. We equip it with the inclusion relation. The order complex of this poset is the complex of free factors defined in \cite{HV1998}.
Notice that $\FC_n$ is a subposet of $\factorposet$ via $A\mapsto \{A\}$, with order given by refinement in the latter.
\item $\CB(F_n)=\CB_n\subset \factorposet$ consists of all $d\in \factorposet$ that are basis compatible. That is,~there exists a basis of $F_n$ such that each $A\in d$ is generated by a subset of this basis. We equip it with the inclusion relation from $\factorposet$.
Note that $\CB(F_n)$ is isomorphic to the face poset of a simplicial complex, so we can also view $\CB(F_n)$ as a simplicial complex.
\item $\PD(F_n) = \PD_n \subset \factorposet$ consists of all $d\in \factorposet$ such that the elements of $d$ form a partial decomposition. Recall that a collection of free factors $d=\ls A_1,\ldots, A_k \rs$ is a partial decomposition if the natural map $$ A_1\ast \cdots \ast A_k \to F_n$$ is injective with image a free factor.   We equip $ \PD_n$ with the refinement relation from  $\factorposet$. 
In general, $\PD_n$ is not a simplicial complex.
We write $ \Decomp(F_n) = \Decomp_n \subset \PD(F_n)$ for the subposet given by all full decompositions, that is, partial decompositions that span $F_n$.
\end{itemize}
The order complex of $\FC_n$ is the complex of free factors defined by Hatcher--Vogtmann \cite{HV1998}. The posets $\CB(F_n)$ and $\PD(F_n)$ generalize related complexes for rings defined by Rognes \cite{ROGNES1992813} and Hanlon--Hersh--Shareshian \cite{HHS}, respectively, and they already appeared in \cite{Brueck2024, Piterman2024}.
Note both $\CB(F_n)$ and $\PD(F_n)$ contain only \emph{finite} collections of free factors.
When $n$ is fixed, we will sometimes drop the subscript $n$ or $F_n$ from the notation and just write $\FC$, $\CB$, $\PD$, and $\Decomp$.

\begin{example}
Every maximal chain in $\PD(F_2)$ is of the form
\[  \{ \Span{e_1} \} \leq \{ \Span{e_1}, \Span{e_2} \}, \]
where $\{e_1,e_2\}$ is a basis of $F_2$.
In particular, the refinement ordering is just the inclusion ordering.

On the other hand, the maximal simplex of $\CB(F_2)$ corresponding to $\{e_1,e_2\}$ is given by
\[ \{\Span{e_1}, \Span{e_2}\}.\]
In this case, we see that $\CB(F_2) = \PD(F_2)$, and they have the homotopy type of a wedge of $1$-spheres since they are $1$-dimensional and connected (see discussion above Proposition 6.7 in \cite{Piterman2024}; cf. \cite{Sadofschi}).
\end{example}

The following is a special case of \cite[Theorem 2.9]{Brueck2024}.

\begin{theorem}
\label{thm:CBandPD}
Let $n\geq 1$.
There is an $\Aut(F_n)$-equivariant map $\psi: \sd \PD(F_n)\to \CB(F_n)$ that is a homotopy equivalence.
\end{theorem}

\subsection{Geometric complexes}
\label{sec:geom_complexes}

We will also work with sphere system complexes, which first appeared in work of Whitehead \cite{Whi:CertainSetsElements,Whitehead1936} and were later generalized by Hatcher \cite{Hat:Homologicalstabilityautomorphism}.
Let $M_{n,k}$ denote the manifold obtained as the $n$-fold connected sum of $S^2\times S^1$ after removing the interior of $k$ disks (i.e., with $k$ boundary components).
A sphere system in $M_{n,k}$ is a finite set of isotopy classes of disjointly embedded $2$-spheres which are not isotopic to a boundary sphere and do not bound a $3$-disk.
We denote by $\sphere(M_{n,1})$ the simplicial complex whose $p$-simplices are the sphere systems in $M_{n,k}$ that have size $p+1$ (i.e.~consist of $p+1$ different isotopy classes of embedded $2$-spheres) and where the face relation is given by containment.

If $\sigma = \{ [S_1],\ldots, [S_r]\}\in \sphere(M_{n,k})$ is a sphere system, we can choose representatives $S_i$ and open tubular neighborhoods $S_i\subset T_i$ in the interior of $M_{n,k}$ such that $T_i\cap T_j = \emptyset$ for $i\neq j$. We will write $M_{n,k}- \sigma$ for $M_{n,k} \setminus \bigcup_i T_i$. Indeed, $M_{n,k}- \sigma$ is a (possibly disconnected) 3-manifold with boundaries and its diffeomorphism type depends only on $\sigma$, not on the choices of representatives $S_i$ and open tubular neighborhoods $T_i$. Each connected component $N$ of $M_{n,k}- \sigma$ is diffeomorphic to $M_{n',k'}$ for some $n',k'\in \mbN$
See Figure \ref{fig:completionVSsubmanifold}.
\begin{figure}
    \centering
    \includegraphics{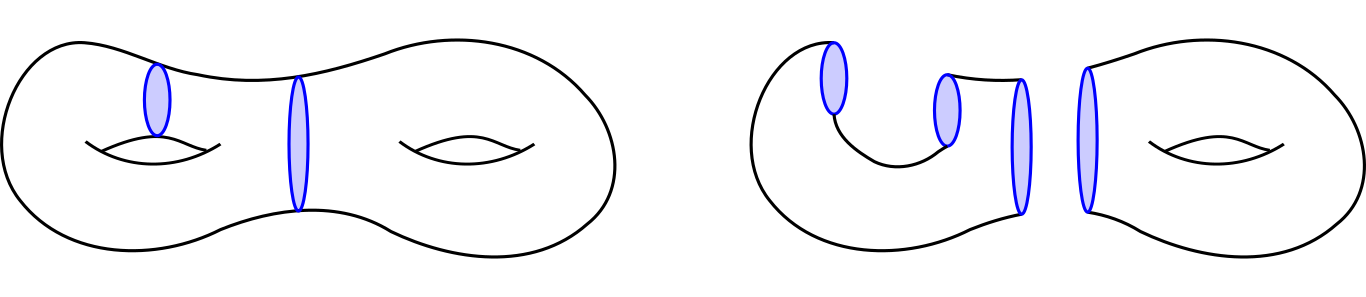}
    \caption{Left: Manifold $M_{2,0}$ with embedded sphere system $\sigma = \ls [S_1], [S_2]\rs$ in blue. Right: The manifold $M_{2,0}-\sigma$ has two connected components, diffeomorphic to $M_{0,3}$ and $M_{1,1}$.}
    \label{fig:completionVSsubmanifold}
\end{figure}
One can show that there is an embedding $\sphere(N)\hookrightarrow \lk_{\sphere(M_{n,k})}(\sigma)$ whose image is given by all isotopy classes of sphere systems that have representatives lying in $N$ and are disjoint from $\sigma$.

A sphere system $\sigma$ has a \emph{dual graph} that is defined as follows (cf.~\cite[Appendix]{Hat:Homologicalstabilityautomorphism}).
The vertices of $\Gamma(\sigma)$ are the connected components of $M_{n,k} - \sigma$.
Each $[S]\in \sigma$ determines an edge in $\Gamma(\sigma)$ joining the connected components on the two sides of the sphere $S$.
Note that this graph may have multiple edges and loops. 

We now recall a way in which $\Aut(F_n)$ naturally acts on $\sphere(M_{n,1})$. For $M$ an orientable manifold with boundary, let $\Diff(M)$ denote the group of orientation-preserving diffeomorphisms that fix a neighborhood of the boundary pointwise and let $\MCG(M)$ denote the associated mapping class group, that is, the quotient of $\Diff(M)$ by the isotopy relation. The work of Laudenbach \cite{Laudenbach} implies that the action of $\Diff(M_{n,1})$ on $\pi_1(M_{n,1}) \cong F_n$ induces a surjection $$\MCG(M_{n,1}) \to  \Aut(F_n),$$ and that the action of $\MCG(M_{n,1})$ on $ \sphere(M_{n,1})$
factors through $\Aut(F_n)$. In particular, $\Aut(F_n)$ acts on $\sphere(M_{n,1})$.

\section{The poset \texorpdfstring{$\Kfour$}{FC xCB Decomp} and the equivalence \texorpdfstring{$\varphi$}{varphi}}
\label{FCcbD}

From now on, we drop the subscript-$n$ notation and assume that our posets and simplicial complexes are defined for a fixed free group $F_n$ of rank $n$.

Let  $\FC^+ = \FC\cup \{F_n\}$, and write $(\FC^+)^{\op}$ for the poset of non-trivial (but possibly improper)
free factors of $F_n$ with order given by reverse inclusion, i.e.~$A\leq B$ if $A\supseteq B$. 
Let $\Decomp^- = \Decomp \cup \{\emptyset\}$ be the poset $\Decomp$ of (proper) free factor decompositions of $F_n$ where we have added a unique minimal element $\emptyset$.
We take the subposet of $\left( (\FC^+)^{\op}\times \Decomp^- \right) \setminus (F_n,\emptyset)$ of ``basis compatible'' elements, defined by
\begin{equation}
\label{eq:defKfour}
\Kfour := \{ (A,d)\in \FC\times \Decomp \tq \{A\}\cup d\in \CB\} \cup\{ (A,\emptyset) \tq A\in \FC\} \cup \{ (F_n,d)\tq d\in \Decomp\}.
\end{equation}

We will use the following characterization for basis compatibility between free factors and decompositions.

\begin{lemma}
\label{lem:free_factors_decompositions_common_bases}
Let $d\in \Decomp = \Decomp(F_n)$.
For a free factor $A<F_n$, define $I_{A,d} = \{A\cap D\tq D\in d\} \setminus \{ \{1\}\}$ to be the set of all non-trivial intersections of elements from $d$ with $A$.
Then the following hold:
\begin{enumerate}
    \item We have $\gen{I_{A,d}} \cong *_{B\in I_{A,d}} B$
     and $I_{A,d} \in \PD(A) \cup \{ \{A\},\emptyset\}$.
     \item If $s=(A_1 < \cdots < A_k)$ is a chain of free factors of $F_n$, then $s\cup d\in \CB(F_n)$ if and only if for all $A\in s$, we have $I_{A,d} \in \Decomp(A) \cup \{ \{A\} \}$.
\end{enumerate}
\end{lemma}

\begin{proof}
Item (1) is clear since the intersection of free factors $A\cap D$ is a free factor of both $A$ and $D$, and $d$ is a full free factor decomposition of $F_n$.

For Item (2), note that if there is a common basis $\mcB$ for $s\cup d$ then the intersection $A\cap D$, with $A\in s$ and $D\in d$, is spanned by $A \cap D \cap \mcB$.
As $D$ runs over the elements of $d$, we get to
\[ A = \gen{A\cap \mcB}  = \gen{ A\cap \big( \bigcup_{D\in d} D\cap\mcB\big) } = \gen{ \bigcup_{D\in d}  A\cap D\cap\mcB} = \gen{I_{A,d}}.\]
Thus $I_{A,d}\in \Decomp(A)\cup \{ \{A\}\}$.

For the converse, note that
\[ I_{A_1,d} < \cdots < I_{A_k,d} < d\]
is a chain in $\PD(F_n)$.
Hence, by Lemma 2.7 of \cite{Brueck2024}, there is a common basis for this chain.
As $\gen{I_{A_i,d}} = A_i$, we conclude that $s\cup d\in \CB(F_n)$.
\end{proof}

Using the above, we prove the following.

\begin{theorem}
\label{thm:he_kfour}
There is an $\Aut(F_n)$-equivariant order-preserving map $\varphi: \sd\PD \to \Kfour$ that is a homotopy equivalence.
\end{theorem}

\begin{proof}
For an element $x\in \sd\PD$, we write $x_f = \{ s\in x\tq s\notin \Decomp\}$ , and $x_d = x\setminus x_f = \{ d\in x\tq d\in \Decomp\}$.
Thus $x = x_f \cup x_d$.
We define:
\[ \varphi_1(x) = \begin{cases}
    \min_{s\in x_f} \gen{s} & x_f\neq\emptyset,\\
    F_n & x_f = \emptyset.
\end{cases}\]
and
\[ \varphi_2(x) = \begin{cases}
    \max x_d & x_d\neq\emptyset,\\
    \emptyset & x_d = \emptyset.
\end{cases}\]
Let $\varphi(x) = (\varphi_1(x),\varphi_2(x))$.
The map $\varphi$ is well-defined since $\varphi(x) \neq (F_n,\emptyset)$, and $\varphi_1(x),\varphi_2(x)$ have a common basis.
Note that if $x \subseteq y$, then $x_f\subseteq y_f$, and therefore $\varphi_1(x) \supseteq \varphi_1(y)$.
On the other hand, $x_d \subseteq y_d$ implies that $\varphi_2(x) \leq \varphi_2(y)$.
Thus, $\varphi$ is an order-preserving poset map.
Also note that $\varphi$ is $\Aut(F_n)$-equivariant.

Next, we show that $\varphi$ is a homotopy equivalence by applying Quillen's fiber theorem.
Let $(A,d)\in \Kfour$.
We prove that $\varphi^{-1}(\Kfour_{\leq (A,d)})$ is contractible.
Notice that $x\in \varphi^{-1}(\Kfour_{\leq (A,d)})$ if and only if $\varphi_1(x) = \varphi_1(x_f) \supseteq A$ and $\varphi_2(x) = \varphi_2(x_d) \leq d$.
Therefore, $\varphi^{-1}(\Kfour_{\leq (A,d)}) = \sd(\Fdecomp \cup \Fpartdecomp)$, where
\begin{align*}
    \Fdecomp & = \Decomp_{\leq d}, \ \text{ and }\\
    \Fpartdecomp & = \{s\in \PD \setminus \Decomp \tq \gen{s} \supseteq A\}.
\end{align*}
Let $P := P_1\cup P_2$.
Note that if $d\neq\emptyset$, then $\Fdecomp$ and $P_{\leq d}$ are contractible.

We have that $P = P_{\leq d} \cup P_2$ is a union of downward-closed subposets.
If $A = F_n$, then $\Fpartdecomp = \emptyset$, and hence $P = \Fdecomp$ is contractible.
Thus we can assume that $A\neq F_n$.
In such a case, the poset $\Fpartdecomp$ is contractible via the following zig-zag inequality of poset maps:
\[ s \leq \{ \gen{s} \} \geq \{ A\}. \]
Note that all these terms lie in $\Fpartdecomp$.
If $d = \emptyset$, then $\Fdecomp = \emptyset$ and so $P = \Fpartdecomp$ is contractible.

Next, we prove that if $d\neq\emptyset$ and $A\neq F_n$, then $P_{\leq d} \cap \Fpartdecomp$ is contractible as well.
Here
\[P_{\leq d}\cap \Fpartdecomp = \{ s\in \PD \setminus \Decomp \tq \gen{s}\supseteq A \text{ and } s\leq d\}.\]
Write $a_d = \{A\cap D \tq D\in d\}$.
Since $A,d$ are basis compatible, by \cref{lem:free_factors_decompositions_common_bases} we conclude that $\gen{a_d} =A$ and $a_d\leq d$, so $a_d\in P_{\leq d}\cap \Fpartdecomp$.
For $s\in P_{\leq d}\cap \Fpartdecomp$, let
\[ f(s)=\{ \gen{S\in s\tq S\leq D} : D\in d\} = \{ \gen{s}\cap D\tq D\in d\}.\]
Again, by \cref{lem:free_factors_decompositions_common_bases}, we have that $s\leq f(s) \leq d$ and $f(s) \in P_{\leq d}\cap \Fpartdecomp$ for $s\in P_{\leq d}\cap \Fpartdecomp$.
Thus we have an order-preserving map $f:P_{\leq d}\cap \Fpartdecomp \to P_{\leq d}\cap \Fpartdecomp$.
Moreover, as $\gen{s}\supseteq A$, we have
\[ s \leq f(s) \geq a_d.\]
Hence $P_{\leq d}\cap \Fpartdecomp$ is contractible.

We have shown that for $A\neq F_n$ and $d\neq\emptyset$, $P$ is also contractible as it is the union of two contractible downward-closed subposets whose intersection is contractible.
Therefore $\varphi^{-1}(\Kfour_{\leq (A,d)}) =  \sd P$ is contractible, for any pair $(A,d)\in \Kfour$.
By Quillen's fiber theorem, $\varphi$ is a homotopy equivalence.
\end{proof}

\begin{remark} One can easily define a version of $\Kfour$ for rings and the above arguments go through for sufficiently nice rings (e.g.~local rings or Dedekind domains). In fact, \cite{Brueck2024} defined a general framework that includes free groups, modules over sufficiently nice rings, as well as other examples such as matroids. In this generality, they proved $\CB \simeq \PD$. Our equivalence $\PD \simeq \Kfour $ can be adapted to this generality. 
\end{remark}

\section{The geometric complex \texorpdfstring{$\geomPD$}{PDg} and the map \texorpdfstring{$\alpha\colon \geomPD\to \Kfour$}{alpha}} \label{geoComplex}

In this section, we define a poset $\geomPD$, which we regard as a geometric version of $\Kfour$ and we give a map $\alpha\colon \geomPD\to \Kfour$.

Throughout this section, we fix $n$ and write $M = M_{n,1}$ and $\sphere = \sphere(M_{n,1})$. We fix a point $v_0$ in the boundary of $M$ and an identification $F_n \cong \pi_1(M,v_0)$.
If $\delta$ is a sphere system in $M$, we call the connected component of $M \sm \delta$ that contains $v_0$ the \emph{basepoint component} and denote it by $\boundcomp(\delta)$ (or just $\boundcomp$).
Recall that $\Gamma(\delta)$ denotes the dual graph of a sphere system $\delta\in \sphere$, as defined in \cref{sec:geom_complexes}. By definition, $\boundcomp(\delta)$ is a vertex of $\Gamma(\delta)$.

A vertex $v$ of a graph $G$ is a \emph{cut vertex} if $|G|\setminus \{v\}$ is disconnected, where $|G|$ denotes the geometric realization of the graph.
We say that the sphere system \emph{$\delta$ is cut}, if $\boundcomp(\delta)$ is simply connected and $\boundcomp(\delta)$ is a cut vertex of $\Gamma(\delta)$. For example, the system $\tau$ in the manifold $M$ at the top left of \cref{fig:MCcovering} is cut.

\begin{definition}
\label{def:geomPD}
Let $\geomPD = \geomPD_n$ be the subposet of $\sphere$ consisting of non-empty sphere systems that either are cut or have a non simply connected basepoint component.
\end{definition}

In the above definition,  we consider $\sphere$ as a poset via its face poset, so $\geomPD$ is a certain subposet of $\sphere$ (which does not contain the empty simplex). In fact, we will see in \cref{prop:geomPD_simplicial} that $\geomPD$ is a sub\emph{complex} of $\sphere$, but this will not be needed for now.
In what follows, we usually denote general elements in $\geomPD$ by $\delta$, write $\tau$ for systems that are cut, and $\sigma$ for systems whose basepoint component is not simply connected.

Our goal is to prove the following theorem:

\begin{theorem}
\label{thm:he_geomPD_kfour}
There is an $\Aut(F_n)$-equivariant order-preserving map $\alpha: \sd \geomPD \to \Kfour$ that is a homotopy equivalence.
\end{theorem}

Recall that the action of $\Aut(F_n)$ on $\geomPD$ comes from the action of the mapping class group. In order to define the map $\alpha: \sd \geomPD \to \Kfour$, we need to explain how we can associate free factors and decompositions of $F_n$ to sphere systems in $\geomPD$. The following lemma will let us map sphere systems whose basepoint component is not simply connected to the fundamental group of the basepoint component.

\begin{lemma}
\label{lm:sigma_to_factor}
Let $\sigma \in \sphere$ be a non-empty sphere system. Then
$$\pi_1(\boundcomp(\sigma),v_0) \to \pi_1(M,v_0)$$ is injective. Additionally,
\begin{equation*} 
    \pi_1(\sigma) := \im \pi_1(\boundcomp(\sigma),v_0) \subseteq \pi_1(M,v_0) = F_n
\end{equation*}
is a proper free factor. Furthermore, if $\sigma \subseteq\sigma'\in \sphere$, then $\pi_1(\sigma)\supseteq \pi_1(\sigma')$.
\end{lemma}
\begin{proof}
    This is an easy consequence of van Kampen's theorem.
\end{proof}

The next lemma shows that we can associate to every $\tau \in \geomPD$ that is cut a non-trivial decomposition, that is, an element of $\Decomp$.

\begin{lemma}
\label{lm:tau_to_decomp}
Let $\tau \in \geomPD$ be cut.
If $C$ is a connected component of $|\Gamma(\tau)| \setminus \{ \boundcomp(\tau)\}$, write $\tau_C$ for the subset of spheres in $\tau$ that give the edges of $|\Gamma(\tau)| \setminus C$.
Then
\begin{equation}
\label{eq:defDecompTau}
\decompMap(\tau) := \{ \, \pi_1(\tau_C) \tq C \text{ is a connected component of }|\Gamma(\tau)|\setminus \{\boundcomp(\tau)\} \, \}
\end{equation}
is a non-trivial free factor decomposition of $F_n$.
\end{lemma}

\begin{proof}
Let $C$ be a connected component of $|\Gamma(\tau)|\setminus \{\boundcomp(\tau)\}$ and let $M_C = M_0(\tau_C)$ be the basepoint component of $M -\tau_C$.
For $C\neq C'$, note that $M_C\cap M_{C'} = \boundcomp(\tau)$.
The $M_C$ give a covering of $M$ (which can be made open by removing the boundary components of each $M_C$ that do not contain the basepoint), each $M_C$ contains the basepoint $v_0$, and they pairwise intersect in a simply connected subspace. Hence by van Kampen's theorem, we conclude that $\decompMap(\tau)$ is a non-trivial free factor decomposition of $F_n$.
See \cref{fig:MCcovering}.
\begin{figure}
\includegraphics{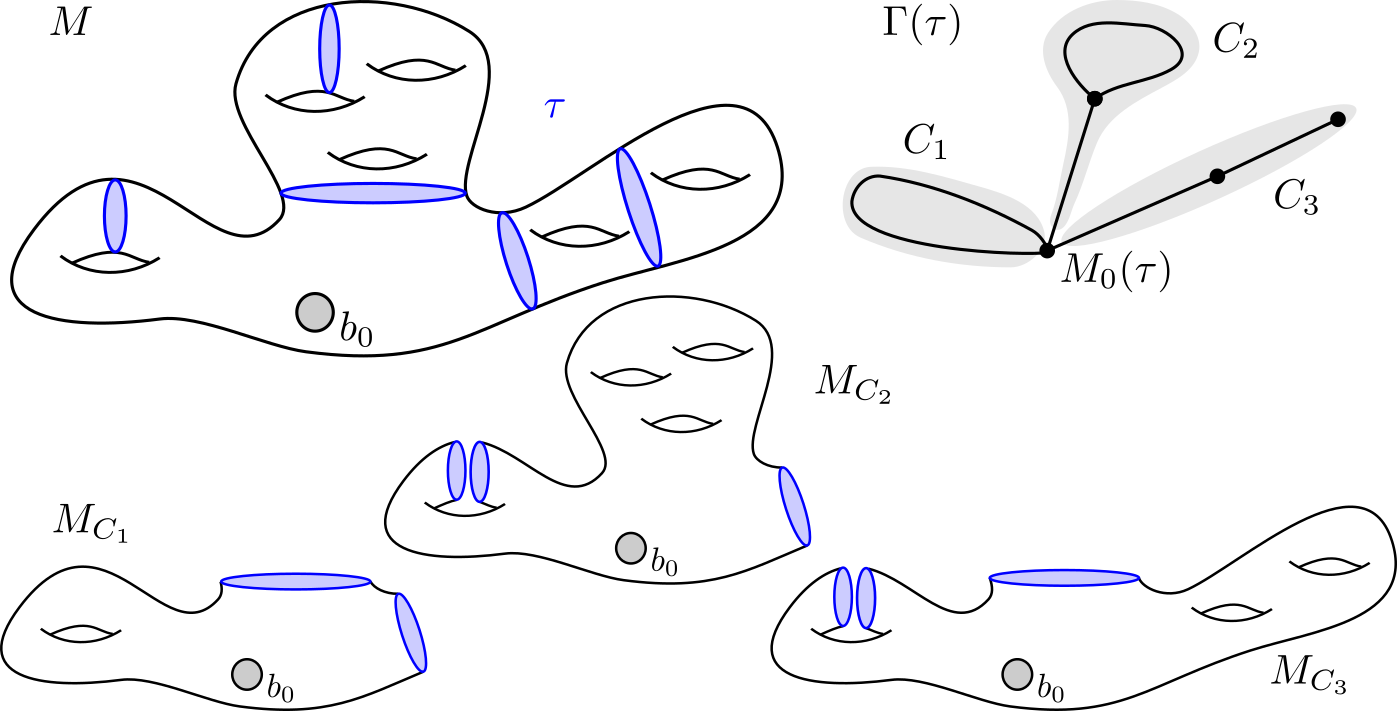}
\caption{Top left and right: A sphere system $\tau\in \geomPD$ (in blue) that is cut and the graph $\Gamma(\tau)$ with $C_1$, $C_2$, $C_3$ the connected components of $|\Gamma(\tau)|\setminus \{\boundcomp(\tau)\}$. Below: The open cover $\ls M_{C_1}, M_{C_2}, M_{C_3}\rs$ from the proof of \cref{lm:tau_to_decomp}.}
\label{fig:MCcovering}
\end{figure}
\end{proof}

\begin{corollary}
\label{it:tau_and_d_tau}
Let $\tau, \tau'$ be two cut sphere systems with $\tau\subseteq\tau'$. Then $\decompMap(\tau)\leq \decompMap(\tau')$.
Furthermore, if they have the same spheres adjacent to the basepoint component, then $\decompMap(\tau) = \decompMap(\tau')$.
\end{corollary}

\begin{proof}
We show that adding spheres to $\tau$ can merge connected components in the corresponding dual graph, thus mapping to a coarser full decomposition in view of \cref{lm:tau_to_decomp}:
Let $[S]\in \tau'\setminus \tau$, so $[S]\in \lk_\sphere(\tau)$.
As explained at the end of \cref{sec:geom_complexes}, we can view $\sphere(\boundcomp(\tau))$ as a subcomplex of $\lk_\sphere(\tau)$.
First, assume that $\ls [S]\rs \in \lk_\sphere(\tau)$ is not contained in $\sphere(\boundcomp(\tau))$. In this case, the dual graph of $\tau\cup \{[S]\}$ has the same number of connected components as $\Gamma(\tau)$ after removing the basepoint component from its geometric realization.
On the other hand, if $\ls [S]\rs  \in \sphere(\boundcomp(\tau)) \subseteq \lk_{\sphere}(\tau)$, then $\boundcomp(\tau) - \{[S]\}$ has two components that each contain at least two boundary components. (This is true because $\tau$ is cut, so $M_0(\tau)\cong M_{0,k}$ for some $k\geq 2$, see \cref{fig:growingTau}.)
In this case, $\decompMap(\tau \cup \{[S]\})$ is obtained from $\decompMap(\tau)$ by merging those free factors that lie on the side of $S \subseteq \boundcomp(\tau)$ that does not include the boundary component of $M$. Such a merge can be trivial, see the orange sphere in \cref{fig:growingTau}, but it always leads to a decomposition that is equal to $\decompMap(\tau)$ or coarser.

\begin{figure}
\includegraphics{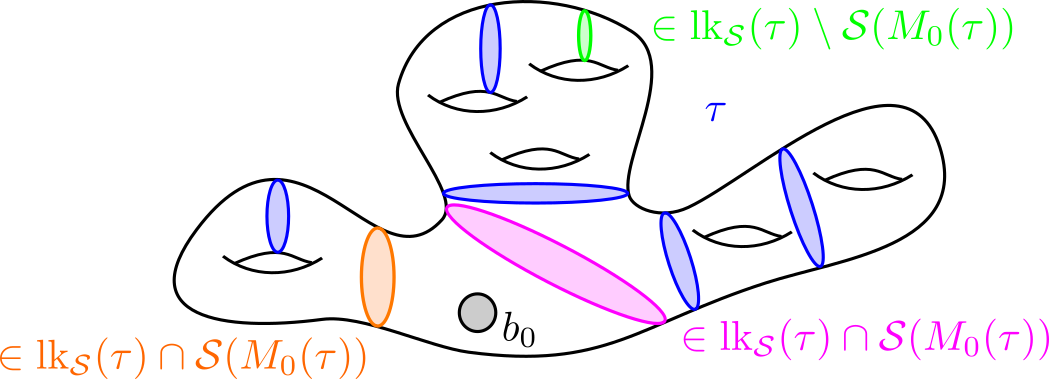}
\caption{The sphere system $\tau$ (in blue) from \cref{fig:MCcovering}, and additional spheres (green, orange and pink) in the link of $\tau$. Adding the green or the orange sphere to $\tau$ gives rise to the same decomposition, while adding the pink sphere gives rise to a coarser decomposition.}
\label{fig:growingTau}
\end{figure}

For the Furthermore part, let $\tau_0\subseteq \tau$ be the set of spheres that are adjacent to the basepoint component. Then $\Gamma(\tau_0)$ is obtained from $\Gamma(\tau)$ by collapsing all edges that are not adjacent to the basepoint vertex $M_0(\tau)$.
This implies that every component $C$ of $|\Gamma(\tau)| \setminus \{ \boundcomp(\tau)\}$ corresponds to a unique component $C_0$ of $|\Gamma(\tau_0)|\setminus \{M_0(\tau_0)\}$ such that $(\tau_0)_{C_0} = \tau_C\cap \tau_0$.
It is also not hard to see that $\pi_1(\tau_C\cap \tau_0) = \pi_1(\tau_C)$ because $\tau_C\cap \tau_0$ are the spheres of $\tau_C$ adjacent to the basepoint component. Putting this together, we get $\decompMap(\tau) = \decompMap(\tau_0)$. So if $\tau_0 = \tau'_0$, then also $\decompMap(\tau) = \decompMap(\tau')$. 
\end{proof}

We are now ready to define the map $\alpha$ from \cref{thm:he_geomPD_kfour}.
For $x\in \sd\geomPD$, write $x_s = \{\sigma\in x\tq \pi_1(\sigma)\neq 1\}$ and $x_t = \{\tau \in x\tq \tau \text{ is cut}\}$.
We define a map $\alpha$ by
\begin{equation}\label{eq:alpha_def}
\begin{aligned}
\alpha : \sd \geomPD &\longrightarrow \Kfour \\
x &\longmapsto \bigl( \pi_1(\max x_s),\, \decompMap(\max x_t) \bigr).
\end{aligned}
\end{equation}
Here, if $x_t = \emptyset$, we set $\decompMap(\max x_t) = \emptyset$, and for $x_s = \emptyset$, we set $\pi_1(x_s) = F_n$.

\begin{lemma}
\label{lem_alpha_poset_map}
The map $\alpha$ is a well-defined order-preserving map of posets that is $\Aut(F_n)$-equivariant.
\end{lemma}

\begin{proof}
From \cref{lm:sigma_to_factor} and \cref{it:tau_and_d_tau} it follows that
$\alpha(x) \in (\FC\cup \{F_n\})^{\op} \times (\Decomp \cup \{\emptyset\})$ for all $x\in \sd\geomPD$ and that $\alpha$ is order preserving.

We will now show that the image of $\alpha$ lies actually in $\Kfour$.
That is, for $x\in \sd\geomPD$, $\pi_1(\max x_s)$ and $\decompMap(\max x_t)$ have a common basis, and $\alpha(x)\neq (F_n,\emptyset)$.
Indeed, since $x\neq \emptyset$, either $x_s$ or $x_t$ is non-empty, and hence $\alpha(x) \neq (F_n,\emptyset)$.
By \cref{lem:free_factors_decompositions_common_bases}(2), we only need to prove that if $\sigma \subseteq \tau \in \geomPD$ where $\sigma$ has non simply connected basepoint component and $\tau$ is cut, we have $A = \gen{I_{A,d}}$, where $A := \pi_1(\sigma)$ and $d:= \decompMap(\tau)$.
For $D\in d$, let $C_D$ be the connected component of $|\Gamma(\tau)|\setminus \{M_0(\tau)\}$ such that $\pi_1(\tau_{C_D}) = D$ (see \cref{lm:tau_to_decomp}), and let $\sigma_D = \tau_{C_D} \cup \sigma$.
Then $A\cap D = \pi_1(\sigma_D)$.
Since $\boundcomp(\sigma)$ is covered by the submanifolds $\boundcomp(\sigma_D)$, $D\in d$, where two different of them intersect in a simply connected submanifold, we conclude that $A = \gen{A\cap D\tq D\in d}$, that is, $A = \gen{I_{A,d}}$.

Finally, it is clear that $\alpha$ is $\Aut(F_n)$-equivariant.
\end{proof}

Showing that $\alpha$ is a homotopy equivalence requires some more work and will be done in the next section.

\section{The map \texorpdfstring{$\alpha$}{alpha} is a homotopy equivalence}
\label{sec:alpha_equiv}

As in the previous section, we fix $n\geq 1$, write $M=M_{n,1}$ and $\sphere = \sphere(M)$.
We also fix an identification $\pi_1(M,v_0)\cong F_n$, and drop the subscript notation for $\geomPD = \geomPD_n$, $\FC = \FC_n$, and $\Decomp=\Decomp_n$.

The aim of this section is to show that the map $\alpha$ defined in \cref{eq:alpha_def} is a homotopy equivalence.
To prove this, we show that its fibers are contractible:

\begin{proposition}
\label{prop:alpha_he}
For all $(A,d)\in \Kfour$, the fibre $\alpha^{-1}\left(\Kfour_{\leq (A,d)}\right)$ is contractible.
\end{proposition}

From now on, we fix $(A,d)\in \Kfour$, and write $d = \{D_1,\ldots,D_k\}$.
The following lemma is an important technical ingredient in the proof of \cref{prop:alpha_he}.

\begin{lemma}
\label{lem:min_sphere_system}
If $d\neq \emptyset$, then there is a unique $\tau_d\in \geomPD$ such that
\begin{enumerate}
\item $|\tau_d| = |d| = k$;
\item every sphere of $\tau_d$ is separating;
\item \label{it:comps_tau_d}$M-\tau_d$ has $k+1$ components. The basepoint component is a 3-ball with $k+1$ boundary components (one of which contains the basepoint). Every other component has exactly one boundary component, which corresponds to the sphere in $\tau_d$ that is adjacent to it and to the basepoint component; their fundamental groups are $D_1, \ldots, D_k$ under the identification $\pi_1(M)\cong F_n$. In particular, $\decompMap(\tau_d) = d$.
\end{enumerate}
\end{lemma}

\begin{proof}
Let $n_i$ denote the rank of $D_i$. View $M = M_{n,1}$ as $M_{0,k+1}$ with $M_{n_1,1}, \ldots, M_{n_k,1}$  glued along $k$ out of the $k+1$ boundary spheres (the remaining one having the basepoint). Call these boundary spheres $S_1,\ldots,S_k$. For $i=1,\ldots,k$, let $B_i$ denote the fundamental group of the complement of $S_1,\ldots,\hat{S_i},\ldots,S_k$. Then work of Hepworth \cite[Proposition 3.5]{Hepworth} implies that the set of ordered decompositions of $F_n$ into free factors of ranks $n_1,\ldots,n_k$ is equivariantly isomorphic to the set
\[\Aut(F_n)/(\Aut(F_{n_1}) \times \cdots \times \Aut(F_{n_k}) ).\]
This is a transitive $\Aut(F_n)$-set, and hence, there exists $f \in \Aut(F_n)$ with $f(B_i)=D_i$ for all $i$. Recall that the $\Aut(F_n)$ action on $\sphere$ comes from the action of $\Diff(M)$ on $\pi_1(M) \cong F_n$ and that Laudenbach \cite{Laudenbach} proved that $$\Diff(M) \to \Aut(F_n)$$ is surjective (see \cref{sec:geom_complexes}). By lifting $f$ to a diffeomorphism of $M$, it follows that $\tau_d=\{f([S_1]),\ldots,f([S_d])\}$ has the desired properties since $\ls [S_1],\ldots,[S_k] \rs$ does with $D_i$ replaced by $B_i$. 

Now we will argue that $\tau_d$ is unique. We may assume that $\tau_d=\ls [S_1],\ldots,[S_k] \rs $ and that $M$ is obtained from $M_{0,k+1}$ by attaching $M_{n_i,1}$ along $S_i$ for $i=1,\ldots,k$. Let $\theta_d= \ls [T_1],\ldots,[T_k] \rs$ be another sphere system with the desired properties. Since $M - \tau_d$ and $M - \theta_d$ are homeomorphic $3$-manifolds with boundary, there exists a diffeomorphism $\tilde f: M \to M $ with $\tilde f(S_i)=T_i$ for all $i$. Let $f \in \Aut(F_n)$ be the induced map on the fundamental group.
Let $\tau_d^i=\tau_d \setminus \ls [S_i]\rs$ and let $\theta_d^i=\theta_d \setminus \ls [T_i]\rs$. Since
\[ D_i =\pi_1(M - \tau_d^i  )=\pi_1(M - \theta_d^i)\]
(cf.~\cref{lm:tau_to_decomp}), 
we have $f(D_i)=D_i$ for all $i$. By \cite[Proposition 3.5]{Hepworth}, $f$ is in the image of
\[\Aut(D_1) \times \cdots \times  \Aut(D_k) \to \Aut(F_n).\]
View $M_{n_i,1}$ as the non-basepoint component of $M \setminus S_i$. The inclusion
\[\Aut(D_i) \to \Aut(F_n)\]
coming from the decomposition
\[ F_n=D_1 * \cdots * D_k\]
lifts to the extension by the identity map
\[\Diff(M_{n_i,1}) \to \Diff(M).\]
By definition, elements of these diffeomorphism groups fix the boundary. Since the boundary of $M_{n_i,1}$ is $S_i$, the $\Aut(D_i)$ action on $\sphere$ fixes $S_i$. Thus $f$ fixes $\tau_d$. Hence $\tau_d=\theta_d$.
\end{proof}

We define the following subposets of $\geomPD$ that will be used to show that all fibers of $\alpha$ are contractible.
\begin{align}
\begin{split}
    \label{eq:def_F_1_F_2}
    \Ftau &\coloneqq \{\tau \in \sphere \tq \tau \text{ is cut and } \decompMap(\tau)\leq d\}, \\
    \Fsigma &\coloneqq \{\sigma \in \sphere \tq \pi_1(\sigma)\supseteq A\},\quad \text{ and }\\
    \Fsigma'& \coloneqq \Fsigma\cap \st_\sphere(\tau_d),
\end{split}
\end{align}
where $\st_\sphere(\tau_d)$ denotes the star of the simplex $\tau_d$ in the sphere system complex $\sphere$.
Here $\Fsigma$ and $\Fsigma'$ are subcomplexes of $\sphere$ that we will usually regard as subposets of $\geomPD$.
In the next two subsections, we will study properties of these posets that will afterwards be used to finish the proof of \cref{prop:alpha_he}.

\subsection{The poset \texorpdfstring{$\Ftau$}{Fcut}}

We start with the following lemma, which says that being cut is an upward-closed property in $\sphere$.

\begin{lemma}
\label{lem:cut_vertex_upwards_closed}
    If $\tau\in \sphere$ is cut and $\tau'\supseteq \tau$ is a sphere system containing $\tau$, then $\tau'$ is cut.
\end{lemma}
\begin{proof}
    First observe that by \cref{lm:sigma_to_factor}, the basepoint component of $\tau'$ has trivial fundamental group.
    To see that it defines a cut vertex of $\Gamma(\tau')$, we use that $\Gamma(\tau)$ is obtained from $\Gamma(\tau')$ by a sequence of edge collapses.
    Hence, it suffices to note the following: 
    Let $(G',v'),(G,v)$ be finite graphs with marked vertices $v',v'$ and assume that $G$ is obtained from $G'$ by collapsing an edge $e\in G'$, such that the collapsing map $f:G'\to G$ sends $v'$ to $v$. Then if $v'$ is not a cut vertex of $G'$, also $v$ is not a cut vertex of $G$. In other words: If $|G'|-v'$ is connected, then so is $|G|-v$. This follows from the observation that collapsing an edge in a connected graph again yields a connected graph.
\end{proof}

Recall from \cref{sec:geom_complexes} that if $\delta\in \sphere$ and $N$ is a connected component of $M-\delta$, then $\sphere(N) \subseteq \sphere$.

\begin{lemma}
\label{lem:F_1_comp_with_tau}
If $d\neq \emptyset$, then for all $\tau\in \Ftau$ and $\tau'\subseteq \tau_d$, we have $\tau\cup \tau'\in \Ftau$.
\end{lemma}
\begin{proof}
\begin{figure}
\centering
\includegraphics{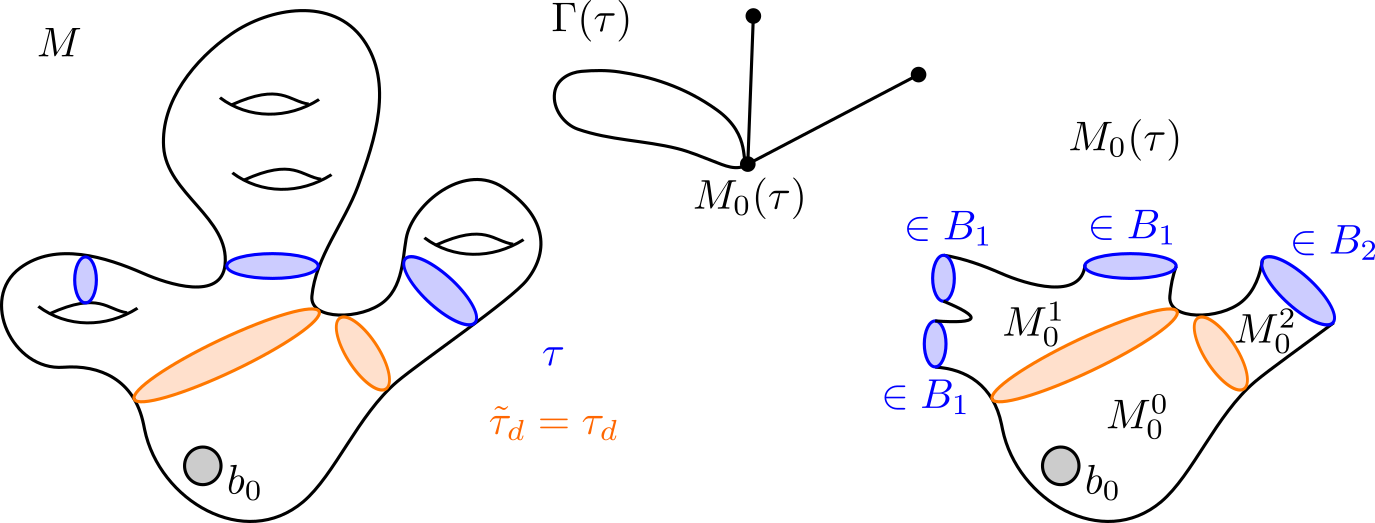}
\caption{On the left: The manifold $M\cong M_{4,1}$, a sphere system $\tau\in \Ftau$ and the system $\tau_d$ for a decomposition $d$ of the form $\ls D_1\cong F_3, D_2\cong F_1 \rs$. On the right: The corresponding decomposition of $M_0(\tau)$ from the proof of \cref{lem:F_1_comp_with_tau}.}
\label{fig:tau_d_and_tau_part}
\end{figure}
We start with the case $\tau'=\tau_d$.
Here, $\boundcomp(\tau)$ is a 3-sphere with one boundary component $b_0$ containing the basepoint and a set $B$ of further boundary components, corresponding to spheres of $\tau$.

To each non-basepoint boundary component $b \in B$, we can associate a unique $D\in d$ as follows:
The component $b$ corresponds to a unique edge in the edge set $E(\Gamma(\tau))$ of the dual graph of $\tau$; each such edge lies in a unique connected component of $|\Gamma(\tau)|\setminus \ls \boundcomp(\tau)\rs$; each such connected component corresponds to a unique $D'\in d'\coloneqq \decompMap(\tau)$ (see \cref{lm:tau_to_decomp}); as $d'\leq d$, each $D'\in d'$ is contained in a unique element of $d$. In summary, this gives a sequence of maps
\begin{equation*}
    B \to E(\Gamma(\tau))\to \ls \text{components of } |\Gamma(\tau)|\setminus \ls \boundcomp(\tau)\rs \rs \to d' \to d.
\end{equation*}
We use this map to partition the set $B$ as $B = B_1\sqcup \ldots \sqcup B_k$, where $d = \ls D_1, \ldots, D_k \rs$ and $B_i$ is the preimage in $B$ of $D_i$ under the composition of the above maps (see \cref{fig:tau_d_and_tau_part}).

In $M_0(\tau)$, we can find a system $\tilde\tau_d$ of $k$ pairwise-disjoint 3-spheres with the following properties: All spheres are separating; $M_0(\tau)-\tilde\tau_d$ has $k+1$ connected components $M_0^0,M_0^1,\ldots, M_0^k$; the component $M_0^0$ contains the boundary component $b_0$ of $M_0(\tau)$ and for $i>0$, the component $M_0^i$ contains exactly the boundary components in $B_i$ plus the component arising from the unique separating sphere of $\tilde\tau_d$.

Using the inclusion $\sphere(M_0(\tau))\hookrightarrow \sphere$, we regard $\tilde{\tau}_d$ as an element in $\sphere$ (see discussion at the end of \cref{sec:geom_complexes}).
It is easy to see that  $\pi_1(\tilde{\tau}_d) = 1$ and its basepoint component is cut, so $\tilde\tau_d\in \geomPD$. 
As each $B_i$ above corresponds to a union of connected components of $|\Gamma(\tau)|\setminus \ls \boundcomp(\tau)\rs$, it follows that each sphere in $\tilde{\tau}_d$ is not only separating in $M_0(\tau)$, but also in $M$. Using van Kampen's theorem and the identification of $d'$ with $\decompMap(\tau)$ from \cref{lm:tau_to_decomp}, we have $\decompMap(\tilde{\tau}_d) = d$. Hence, $\tilde{\tau}_d$ satisfies all properties of $\tau_d$ from \cref{lem:min_sphere_system} and by the uniqueness statement in \cref{lem:min_sphere_system}, we get $\tilde{\tau}_d = \tau_d$.

Now, as every sphere in $\tilde\tau_d=\tau_d$ is either contained in the basepoint component of $M-\tau $ or lies in $\tau$, we have $\tau\cup \tau_d \in \sphere$. By \cref{lem:cut_vertex_upwards_closed}, the system $\tau\cup \tau_d$ is cut.
Furthermore, the basepoint adjacent spheres of $\tau$ are exactly given by $\tau_d$, so by \cref{it:tau_and_d_tau}, we have $\decompMap(\tau\cup \tau_d) = \decompMap(\tau_d)=d$. This implies that $\tau\cup\tau_d\in \Ftau$.

For the general case, suppose that $\tau'\subseteq \tau_d$ is any subset.
Then, by the previous case, we have $\tau\cup \tau'\subseteq \tau\cup \tau_d \in \Ftau$.
As $\sphere$ is a simplicial complex, also $\tau\cup \tau'\in \sphere$, and $\tau\cup \tau'$ is cut because it contains $\tau$ and being cut is upwards closed by \cref{lem:cut_vertex_upwards_closed}. Lastly, by \cref{it:tau_and_d_tau}, we have $\decompMap(\tau\cup \tau')\leq \decompMap(\tau \cup \tau_d)= d$. Hence, $\tau\cup \tau'\in \Ftau$.
\end{proof}

\subsection{\texorpdfstring{$\Fsigma$ and $\Fsigma'$}{SA and SA'}}
Next, we collect properties of $\Fsigma$ and $\Fsigma'$.
The complex $\Fsigma$ was defined by Hatcher-Vogtmann \cite{HV1998}, who showed that it is contractible (see also the corrected version \cite{HV2022freefactors}).

\begin{theorem}
[{\cite[Thm. 2.1]{Hat:Homologicalstabilityautomorphism} and \cite[Thm. 2.1]{HV1998}}]
\label{thm:S_A_contractible}
The complexes $\sphere_A$ and $\sphere$ are contractible.
\end{theorem}

Recall that a subcomplex $L$ of a simplicial complex $K$ is full if whenever we have a simplex $\sigma\in K$ whose vertices lie in $L$, then $\sigma\in L$.

\begin{lemma}
\label{lem:S_A_simplicial}
    Both $\Fsigma$ and $\Fsigma'$ are full subcomplexes of $\sphere$.
\end{lemma}

\begin{proof}
    To see that $\Fsigma$ is a full subcomplex, we need to show that if $\sigma, \sigma'\in \Fsigma$ are such that $\sigma\cup \sigma'\in \sphere$, then $\sigma\cup\sigma'\in \Fsigma$.
    In other words, if $\pi_1(\sigma)$ and $\pi_1(\sigma')$ contain $A$, then also $\pi_1(\sigma\cup \sigma')\supseteq A$.
    This follows from \cite[Lemma 2.2]{HV2022freefactors}.
    
    The subcomplex $\Fsigma'$ is defined as the intersection of $\Fsigma$ with $\st_{\sphere}(\tau_d)$. The latter is also a full subcomplex of $\sphere$ since it is a flag complex (see Theorem 3.3 of \cite{gadgilpandit}). The claim then follows as the intersection of full subcomplexes is a full subcomplex again.
\end{proof}

The above lemma implies that a sphere system $\sigma \in \sphere$ lies in $\Fsigma$ (resp. $\Fsigma'$) if and only if for all $[S] \in \sigma$, the vertex $\ls [S]\rs$ lies in $\Fsigma$ (resp. $\Fsigma'$).
As a consequence, we obtain:

\begin{lemma}
\label{lem:existence_tau_d_A}
The face $\tau_{d,A} = \ls [S]\in \tau_d \tq \ls [S]\rs\in \Fsigma \rs$ is the unique maximal subsystem of $\tau_d$ such that $\pi_1(\tau_{d,A}) \supseteq A$.
\end{lemma}

\begin{proof}
This follows from \cref{lem:S_A_simplicial}.
\end{proof}

\begin{lemma}
\label{lem:tau_d_A_empty}
Suppose that $\tau_{d,A} = \emptyset$.
Then for every $\sigma\in \Fsigma'$, we have $\sigma\cup \tau_d\in \Ftau$.
\end{lemma}

\begin{proof}
Let $\sigma\in \Fsigma'$.
As $\sigma\in \st_{\sphere}(\tau_d)$, we have $\sigma\cup \tau_d\in \sphere$ and since $\tau_d$ is cut, by \cref{lem:cut_vertex_upwards_closed}, $\sigma\cup \tau_d$ is cut.
To conclude that $\sigma\cup \tau_d \in \Ftau$, we prove that the spheres in $\sigma\cup\tau_d$ that are adjacent to the basepoint component are exactly those in $\tau_d$, and then invoke \cref{it:tau_and_d_tau}.

\begin{figure}
\centering
\includegraphics{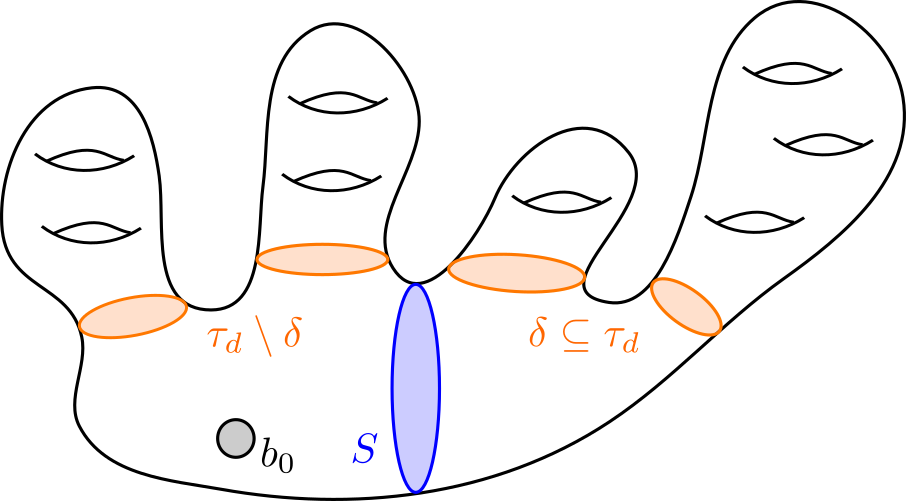}
\caption{The manifold $M$ from the proof of \cref{lem:tau_d_A_empty}. The four orange circles depict $\tau_d$ and the part below the orange circles is $M_0(\tau_d)$. The blue circle depicts a sphere $S$ such that $[S]\in \sphere(M_0(\tau_d))\subset \sphere$.}
\label{fig:fig_tau_d_A_empty}
\end{figure}

Assume that there is an element $[S]\in \sigma$ such that $\ls [S]\rs \in \sphere(\boundcomp(\tau_d)) \subseteq \lk_{\sphere}(\tau_d)$ (see \cref{sec:geom_complexes}).
Recall that $\boundcomp(\tau_d) \cong M_{0,k+1}$ is a 3-ball with $k+1$ boundary components, one containing the basepoint and the remaining $k$ being in one-to-one correspondence with the elements of $\tau_d$ (see \cref{lem:min_sphere_system}).
The manifold $\boundcomp(\tau_d)- \{[S]\}$ has two connected components (each one simply connected).
The component that does not contain the basepoint contains at least one boundary component of $\boundcomp(\tau_d)$ because $S$ does not bound a disc (see \cref{fig:fig_tau_d_A_empty}). Let $\emptyset\neq\delta\subseteq \tau_d$ be the set of spheres of $\tau_d$ corresponding to the boundary components that lie in this connected component of $\boundcomp(\tau_d) - \ls [S] \rs$.
Then it is clear that $\delta \cup \{[S]\} \in \sphere$, and at the level of fundamental groups we get inclusions
\[ \pi_1(\ls [S]\rs) = \pi_1(\delta \cup \{[S]\}) \subseteq  \pi_1(\delta).\]
As $\ls [S]\rs\in \Fsigma$, that is, $\pi_1(\ls [S] \rs)\supseteq A$, we deduce that $\pi_1(\delta)\supseteq A$ and hence $\delta\subseteq \tau_{d,A}$, which is empty by hypothesis.
This contradicts our choice of $[S]$, showing that no such sphere exists.

This implies that the spheres of $\sigma\cup \tau_d$ adjacent to its basepoint component are exactly the ones in $\tau_d$. Hence, $\sigma\cup \tau_d$ is cut and by \cref{it:tau_and_d_tau}, we have $\decompMap(\sigma\cup \tau_d)=\decompMap(\tau_d) = d$, so $\sigma\cup \tau_d\in \Ftau$.
\end{proof}

\begin{lemma}
\label{lem:F_2_cap_st_tau_d_contractible}
If $A\neq F_n$ and $d\neq \emptyset$, the poset $\Fsigma'$ is contractible.
\end{lemma}

\begin{proof}
Let $\tau_d=\{[S_1],\ldots,[S_k]\}$ and write $M_0\cong M_{0,k+1}$ for the basepoint component of $M-\tau_d$ and $M_1,\ldots, M_k$ for the the remaining connected components, where the fundamental group of $M_i$ is identified with $D_i$, $M_i$ is diffeomorphic to $M_{\rk(D_i),1}$, and its boundary component is given by one of the sides of $S_i$ (that we can make these identifications is \cref{it:comps_tau_d} of \cref{lem:min_sphere_system}).
Recall from \cref{sec:geom_complexes} that $\sphere(M_i)$ embeds into $\sphere(M)$, and even in $\lk_{\sphere}(\tau_d)$.
It follows that the star of $\tau_d$ in $\sphere$ is given by the join
\begin{equation*}
\label{eq:star_taud}
  \st_{\sphere}(\tau_d) = \tau_d *  \sphere(M_0) * \sphere(M_1) * \cdots * \sphere(M_k),
\end{equation*}
where we regard $\tau_d$ as a simplicial complex consisting of a single maximal simplex $\tau_d$.

By \cref{lem:S_A_simplicial}, the poset $\Fsigma'$ is a full subcomplex of $\sphere$, so we have
\begin{equation}
\label{eq:SprimeA_join}
    \Fsigma' = \sphere_A\cap \st_{\sphere}(\tau_d) = \sphere_A \cap \tau_d   *  \sphere_A \cap\sphere(M_0) * \sphere_A \cap \sphere(M_1) * \cdots * \sphere_A \cap \sphere(M_k).
\end{equation}
Recall from \cref{lem:free_factors_decompositions_common_bases} that, since $(A,d)\in \Kfour$, the factor $A$ is the free product of $A_1,\ldots,A_k$, where $A_i := A\cap D_i$.
We claim that for $i\geq 1$, we have  
\begin{equation}
\label{eq:no_good_name_for_this}
    \sphere_A \cap \sphere(M_i)= \sphere(M_i)_{A_i}.
\end{equation}
By \cref{thm:S_A_contractible}, the poset $\sphere(M_i)_{A_i}$ is contractible. Hence, \cref{eq:no_good_name_for_this} together with \cref{eq:SprimeA_join} implies that $\Fsigma'$ contains a contractible poset as a join factor (recall that $k\geq 1$ as $d\neq\emptyset$) and hence is contractible itself.

To prove \cref{eq:no_good_name_for_this}, note that if $\delta\in \sphere(M_i)$, then the inclusion $\sphere(M_i)\hookrightarrow\sphere$ lets us see $\delta$ as an element of $\sphere = \sphere(M)$ and its fundamental group as such a sphere system of $M$ is
\[ \pi_1^M(\delta) = \pi_1(M_0 \cup M_i \sm \delta,v_0) * \gen{D_j \tq j\neq i}.\]
Hence, we have $A\subseteq \pi^M_1(\delta)$  if and only if 
\[ \pi_1(M_0\cup M_i \sm \delta,v_0) \supseteq A\cap D_i = A_i.\]
The van Kampen Theorem gives an identification of $\pi_1(M_0\cup M_i \sm \delta,v_0)$ with $\pi_1(M_i \sm \delta,v_i)$, where $v_i$ is a point on the boundary of $M_i$. The latter is the same as $\pi_1^{M_i}(\delta)$, the fundamental group of $\delta$, seen as a sphere system of $M_i$. 
We hence have $A\subseteq \pi^M(\delta)$ (i.e.,~$\delta \in \sphere_A \cap \sphere(M_i)$) if and only if $A_i\subseteq \pi_1^{M_i}(\delta)$ (i.e.,~$\delta \in \sphere(M_i)_{A_i}$). This finishes the proof of \cref{eq:no_good_name_for_this} and hence of the lemma.
\end{proof}

\subsection{Contractibility of the fibers of \texorpdfstring{$\alpha$}{alpha}}

We are now ready to prove \cref{prop:alpha_he}, i.e.,~to show that the fibers of $\alpha$ are contractible.

\begin{proof}[Proof of \cref{prop:alpha_he}]
For the proof, we will regard $\Fsigma$ and $\Fsigma'$ as posets via their face posets.
Let $\tau_d$ be as in \cref{lem:min_sphere_system} and $\Ftau$ and $\Fsigma$ as in \cref{eq:def_F_1_F_2}. The fiber $\alpha^{-1}(\Kfour_{\leq (A,d)})$ is equal to $\sd(\Fsigma\cup \Ftau)$, the poset of chains in $\Fsigma\cup \Ftau$, so it suffices to show that the poset $\Fsigma\cup \Ftau$ is contractible.
As a warm-up, we consider first the cases where $A=F_n$ or $d=\emptyset$. In the former case, $(A,d) = (F_n,d)$ and $d\neq\emptyset$, we have $\Fsigma = \emptyset$, so $\alpha^{-1}(\Kfour_{\leq (A,d)}) = \sd\Ftau$. 
By \cref{lem:F_1_comp_with_tau}, the assignment $\tau\mapsto \tau\cup \tau_d$ gives a monotone poset map on $\Ftau$, hence a deformation retraction from $\Ftau$ to its subposet consisting of all $\tau\in \Ftau$ that contain $\tau_d$. This subposet is contractible because $\tau_d$ is the unique minimal element of it. Hence, the poset $\Ftau$ is contractible as well.
In the second case, namely when $(A,d) = (A,\emptyset)$ and $A\neq F_n$, the poset $\Ftau$ is empty, and therefore $\alpha^{-1}(\Kfour_{\leq (A,d)}) = \sd \Fsigma$, which is contractible by \cref{thm:S_A_contractible}.

We will now assume that $A\neq F_n$ and $d\neq \emptyset$ and define a sequence of two deformation retractions starting at $\Fsigma\cup \Ftau$ and ending at a contractible subposet.

We start by showing that there is a deformation retraction
\begin{equation*}
    \phi_1\colon |\Fsigma\cup \Ftau|\to |\Fsigma' \cup \Ftau|.
\end{equation*}
By  \cref{thm:S_A_contractible} and \cref{lem:F_2_cap_st_tau_d_contractible}, the posets $\Fsigma$ and $\Fsigma'$ are both contractible. Using the Whitehead Theorem \cite[Chapter 4, Theorem 4.4]{Hat:Algebraictopology}, we hence get a deformation retraction $\phi\colon|\Fsigma|\to |\Fsigma'|$.
By \cref{lm:sigma_to_factor}, no element of $\Ftau$ can be less than an element of $\Fsigma$. This implies that we can write
\begin{equation*}
    |\Fsigma\cup \Ftau| = |\Fsigma|\cup |F_{\leq \text{cut}}| \text{, where }F_{\leq \text{cut}} = \ls \delta\in \Fsigma \cup\Ftau  \tq \delta\leq \tau \text{ for some } \tau\in \Ftau\rs.
\end{equation*}
We extend $\phi$ to $|\Fsigma\cup \Ftau|$ by setting it to be the identity on $|F_{\leq \text{cut}}|$. To see that this gives a well-defined deformation retraction to $|\Fsigma' \cup \Ftau|$, we need to verify that $|F_{\leq \text{cut}}|\subseteq |\Fsigma' \cup \Ftau|$ (this in particular implies that $\phi$ restricts to the identity on $|\Fsigma|\cap |F_{\leq \text{cut}}|$).
In fact, we have $F_{\leq \text{cut}}\subseteq \Fsigma' \cup \Ftau$ on the poset level: It is clear that $F_{\leq \text{cut}} \cap \Ftau\subseteq \Ftau$. On other hand, if $\sigma\in F_{\leq \text{cut}}\cap \Fsigma$, then by definition $\sigma\leq \tau$ for some $\tau\in \Ftau$. By \cref{lem:F_1_comp_with_tau}, $\tau$ is contained in $\st_\sphere(\tau_d)$, and hence so is its face $\sigma\subseteq \tau$. This means that $\sigma \in \Fsigma'$.

Next, we define an order-preserving map $\phi_2:\Fsigma'\cup\Ftau \to \Fsigma'\cup \Ftau$ that is monotone and whose image contains a unique minimal element---hence it is contractible.
As such a map defines a homotopy equivalence between $\Fsigma'\cup\Ftau$ and its image, $\Fsigma'\cup \Ftau$ is contractible.

For the definition of $\phi_2$, we use the face $\tau_{d,A}$ from \cref{lem:existence_tau_d_A}, and consider first the case where $\tau_{d,A}$ is non-empty.
In this case, we define a poset map 
\begin{equation*}
    \phi_2 \colon \Fsigma' \cup \Ftau \to  \Fsigma' \cup \Ftau
\end{equation*}
by sending each $\delta$ to $\delta \cup \tau_{d,A}$.
We need to show that this assignment actually maps to $\Fsigma' \cup \Ftau$.
If $\tau\in \Ftau$, then $\tau\cup \tau_{d,A} = \phi_2(\tau)\in \Ftau$ by \cref{lem:F_1_comp_with_tau}.
If on the other hand $\sigma\in \Fsigma'$, then $\sigma\cup \tau_{d,A} = \phi_2(\tau) \in \Fsigma'$ as well.
Indeed, we have $\sigma\cup \tau_{d,A} \in \sphere$ because $\sigma\in \st_\sphere(\tau_d)\subseteq \st_\sphere(\tau_{d,A})$. But $\sigma$ and $\tau_{d,A}$ are also both contained in $\Fsigma'$ and by \cref{lem:S_A_simplicial}, $\Fsigma'$  is a full subcomplex of $\sphere$, so $\sigma \cup \tau_{d,A}\in \Fsigma'$.
This shows that the map $\phi_2$ is well-defined. It is clearly order-preserving and monotone, and its image contains $\tau_{d,A}\in \Fsigma'$ as the unique minimal element.
Therefore, $\Fsigma'\cup \Ftau$ is contractible in this case.

Now suppose that $\tau_{d,A}$ is empty. Here, we define a map
\begin{equation*}
    \phi_2 \colon \Fsigma' \cup \Ftau \to \Fsigma' \cup \Ftau
\end{equation*}
by sending $\delta$ to $\delta \cup \tau_{d}$. Again, we need to check that this is well-defined.
If $\tau\in \Ftau$, then as before, \cref{lem:F_1_comp_with_tau} implies that $\tau\cup \tau_{d} = \phi_2(\tau)\in \Ftau$.
For $\sigma \in \Fsigma'$, we have $\sigma\cup \tau_d = \phi_2(\sigma)\in \Ftau$ by \cref{lem:tau_d_A_empty}.
This shows that $\phi_2$ is well-defined. It is clearly order-preserving and monotone, and its image contains $\tau_{d}\in \Ftau$ as the unique minimal element.
Hence, $\Fsigma'\cup \Ftau$ is contractible.

In any case, we have shown that
\[ \alpha^{-1}( \, \Kfour_{\leq (A,d)} \,) = \sd(\Fsigma\cup \Ftau) \simeq \Fsigma'\cup \Ftau\]
is contractible, which concludes the proof of the proposition.
\end{proof}

\cref{thm:he_geomPD_kfour} now follows from \cref{eq:alpha_def}, \cref{lem_alpha_poset_map}, \cref{prop:alpha_he} and Quillen's fiber theorem.

\section{Connectivity of \texorpdfstring{$\geomPD$}{D(Mn,1)}}
\label{sec:connectivity_D}

As before, we fix $n\geq 1$, write $M=M_{n,1}$, $\sphere = \sphere(M)$ and $\geomPD = \geomPD_n$. We now drop the convention that sphere systems denoted by $\sigma$ satisfy $\pi_1(\sigma) \neq 1$.
In this section, we will prove that $\geomPD$ is $(2n-4)$-connected, which implies \cref{thm:mainCB}. We will follow the strategy of Hatcher--Vogtmann \cite[Proposition 6.2]{HatcherVogtmannCerf}, using \cite[Theorem 4.6]{Aygun2022} in place of \cite[Theorem 3.1]{HatcherVogtmannCerf}. 

Recall that $\Gamma(\sigma)$ denotes the dual graph of a sphere system $\sigma\in \sphere$, where each vertex $x\in \Gamma(\sigma)$ corresponds to a connected component of $M-\sigma$ (see \cref{sec:geom_complexes}). In this section, we regard $\Gamma(\sigma)$ as a vertex-labeled graph, where we label a vertex $x$ by $g(x)$, the rank of the fundamental group of this connected component.
As before, we write $\boundcomp(\sigma)$ for the vertex of $\Gamma(\sigma)$ corresponding to the basepoint component.
We have
\begin{equation}
\label{eq:in_geomPD_graph}
    \sigma\in \geomPD \text{ if and only if }  g(\boundcomp(\sigma))>0 \text{ or } \boundcomp(\sigma) \text{ is a cut vertex of } \Gamma(\sigma).
\end{equation}

Next, we recall the notion of \textit{degree} from \cite{Aygun2022}, which will give a filtration of $\sphere$.

\begin{definition}
\label{def:degree}
The \emph{degree} of a simplex $\sigma\in \sphere$ is defined as
\begin{equation*}
    \deg(\sigma)=\sum_{x \neq \boundcomp(\sigma)} v(x)-2+2g(x),
\end{equation*}
where the sum is taken over all vertices $x$ of $\Gamma(\sigma)$, $v(x)$ denotes the valence of a vertex and $g(x)$ its label.

We define $\sphere^{\leq \thedegree}$ to be the subposet of $\sphere$ consisting of sphere systems $\sigma$ with $\deg(\sigma) \leq \thedegree$. 
\end{definition}

The following is the ``degree theorem'' proved in \cite[Theorem 3.5]{Aygun2022}.

\begin{theorem}
$\sphere^{\leq \thedegree}$ is $(n+\thedegree-2)$-connected.
\end{theorem}

\begin{definition}
Let $\sigma$ be a simplex of $\sphere$. The \emph{pillar} of $\sigma$ is the (possibly empty) face of $\sigma$ consisting of spheres that correspond to non-loop edges of $\Gamma(\sigma)$ that are adjacent to $M_0(\sigma)$. We say that a simplex is a pillar if it is equal to its pillar.
\end{definition}

The following follows from \cite[Corollary 2.3]{Aygun2022}.

\begin{lemma} \label{3.2}
If $\sigma$ and $\tau$ have the same pillar, then $\deg(\sigma)=\deg(\tau)$.
\end{lemma}

\begin{lemma} \label{inD}
Let $\sigma$ and $\tau$ be simplices of $\sphere$. If they have the same pillar and $\sigma \in \geomPD$, then $\tau \in \geomPD$.
\end{lemma}

\begin{proof}
As $\sigma$ and $\tau$ have the same pillars, the graph $\Gamma(\sigma)$ can be obtained from $\Gamma(\tau)$ by a sequence of the following two operations and their inverses: 
\begin{itemize}
    \item collapsing a loop at any vertex $v$ and increasing the label of this vertex by one;
    \item collapsing an edge with distinct endpoints $v\neq w$ that are both different from the basepoint vertex $\boundcomp(\sigma)$ to a new vertex with label $g(v)+g(w)$.
\end{itemize}
We will assume that $\Gamma(\tau)$ is obtained from $\Gamma(\sigma)$ by one of these operations and will check that $\sigma$ is a simplex of $\geomPD$ if and only if $\tau$ is. For this, we use \cref{eq:in_geomPD_graph}.

First consider the case that $\Gamma(\tau)$ is obtained from $\Gamma(\sigma)$ by collapsing a loop. If this loop is not adjacent to the basepoint vertex $\boundcomp(\sigma)$, then clearly $\sigma\in \geomPD$ if and only if $\tau\in \geomPD$. If the loop is adjacent to $\boundcomp(\sigma)$, then we have $\sigma\in \geomPD$ because $\boundcomp(\sigma)$ is a cut vertex (it has a loop adjacent to it) and also $\tau\in \geomPD$ because $g(\boundcomp(\tau)) = g(\boundcomp(\sigma))+1>0$.

Now assume that $\Gamma(\tau)$ is obtained from $\Gamma(\sigma)$ by collapsing an edge that does not meet $\boundcomp(\sigma)$. Then $g(\boundcomp(\sigma)) = g(\boundcomp(\tau))$ and $|\Gamma(\tau)| \setminus \{\boundcomp(\tau) \}$ is obtained from $|\Gamma(\sigma)| \setminus \{\boundcomp(\sigma) \}$ by collapsing an edge. This does not change the set of path components so $\boundcomp(\sigma)$ is a cut vertex in $\Gamma(\sigma)$ if and only if $\boundcomp(\tau)$ is a cut vertex in $\Gamma(\tau)$. Using \cref{eq:in_geomPD_graph}, this shows that $\sigma\in\geomPD$ if and only $\tau\in\geomPD$.  
\end{proof}

\begin{proposition}
\label{prop:geomPD_simplicial}
    If $\sigma$ is a sphere system in $\geomPD$ and $\tau \subseteq \sigma$, then $\tau$ is in $\geomPD$. In other words, $\geomPD$ is a simplicial complex.
\end{proposition}

\begin{proof}
 First assume that $g(\boundcomp(\sigma))>0$. \cref{lm:sigma_to_factor} implies $$g(\boundcomp(\tau)) \geq g(\boundcomp(\sigma))>0,$$ so $\tau \in \geomPD$.

Now assume that $\boundcomp(\sigma)$ is cut. Thus, $\Gamma(\sigma)$ decomposes as a wedge $$\Gamma(\sigma) = \bigvee_{i=1}^k \Gamma_k$$ for a collection of non-trivial connected based graphs $\Gamma_1, \ldots, \Gamma_k$. It suffices to consider the case that $\tau$ is $\sigma$ with a single sphere removed. Note that $\Gamma(\tau)$ is obtained from $\Gamma(\sigma)$ by collapsing a single edge which we will call $e$. Assume that $e$ is an edge of $\Gamma_j$. We have that $$\Gamma(\tau) = \bigvee_{i=1}^k \Gamma_i'$$ with $\Gamma_i'=\Gamma_i$ for $i \neq j$ and $\Gamma_j'$ obtained from by $\Gamma_j$ by collapsing $e$. If $\boundcomp(\tau)$ is cut, then $\tau \in \geomPD$ so assume otherwise. 
Since $\boundcomp(\tau)$ is not cut, we must have  $k=2$ and $\Gamma_j'$ must be a single vertex. Thus $e$ is the only edge of $\Gamma_j$ and $e$ is a loop at $\boundcomp(\sigma)$. This implies that $$g(\boundcomp(\tau))=1$$ and so $\tau \in \geomPD$.
\end{proof}

Note that Aygun and the second author \cite[Corollary 3.3]{Aygun2022} checked that $\sphere^{\leq l}$ is a sub\emph{complex} of $\sphere$.
The following is closely related to Hatcher--Vogtmann \cite[Lemma 5.2 (iii)]{HatcherVogtmannCerf}. Note that their complexes of sphere systems do not include separating spheres so we cannot directly quote their results.

\begin{proposition} \label{5.2}
We have an inclusion of simplicial complexes $\sphere^{\leq n-2} \subseteq \geomPD$.
\end{proposition}
\begin{proof}
    Let $\sigma\in \sphere$ with $\deg(\sigma) \leq n-2$ and let $\tau$ be the pillar of $\sigma$.  By \cref{3.2}, $\deg(\tau) \leq n-2$. By \cref{inD}, it suffices to prove that $\tau$ is in $\geomPD$. Assume that this was not the case, i.e.~that $g(M_0(\tau))=0$ and $M_0(\tau)$ is not a cut vertex of $\Gamma(\tau)$. Since $M_0(\tau)$ is not a cut vertex and $\tau$ is a pillar, the graph $\Gamma(\tau)$ must have exactly two vertices and every edge connects $M_0(\tau)$ to the other vertex $x$. In particular, we have
    \begin{equation*}
        \deg(\tau) = v(x) + 2 g(x)-2 = e+2g(x)-2,
    \end{equation*}
    where $e$ is the number of edges of $\Gamma(\tau)$.    
    Since $g(M_0(\tau))=0$ and $\pi_1(M)$ has rank $n$, we also get
    \begin{equation*}
        n = e+g(x)-1.
    \end{equation*}
    Combining these equations, we obtain
    \begin{equation*}
        \deg(\tau) = n+g(x)-1.
    \end{equation*}
    As the latter is greater than or equal to $n-1$, we get a contradiction. Hence, we have that $g(M_0(\tau))=0$ or $M_0(\tau)$ is a cut vertex, so $\tau\in \geomPD$.
\end{proof}

\begin{theorem}

\label{prop:connectivity_D}
$\geomPD$ is $(2n-4)$-connected.
\end{theorem}
\begin{proof}
Let $\thedegree \leq 2n-4$ and consider a map $f:S^\thedegree \m \geomPD$ from a $\thedegree$-sphere into $\geomPD$. Our goal is to show that $f$ is null-homotopic. By simplicial approximation, we may assume that $f$ is simplicial with respect to some combinatorial simplicial complex structure on $S^\thedegree$ which we call $X$. We will show that $f$ is homotopic to a map that factors through $\sphere^{\leq n-2} \subset \geomPD$. Since $\sphere^{\leq n-2}$ is $(2n-4)$-connected, this will show that $f$ is null-homotopic. 

Let $\themaxdegree=\max_{\Sigma \in X}\deg(f(\Sigma))$ and assume $\themaxdegree>n-2$. Let $\Sigma$ be a maximal-dimensional simplex of $X$ with $f(\Sigma)$ a pillar and $\deg(f(\Sigma))=\themaxdegree$. Let $\boundcomp$ denote the basepoint component of $M - f(\Sigma)$ and let $M_1, \ldots, M_k$ denote the other components. Let $Y(\boundcomp)$ denote the subcomplex of $\sphere(\boundcomp)$ of non-separating sphere systems. The proof of \cite[Theorem 3.5]{Aygun2022} checks that the link $\Link_{X}(\Sigma)$ maps to 
$$Y(\boundcomp) * \sphere(M_1) * \cdots *\sphere(M_k). $$
Moreover, in the proof of \cite[Lemma 3.3]{Aygun2022} it is stated that this join is $(n+\themaxdegree-\dim(f(\Sigma))-3) > (2n-5-\dim(\Sigma))$-connected. As $\Link_{X}(\Sigma)$ is a $(2n-5-\dim(\Sigma))$-sphere, this implies that its image is null-homotopic. Thus, there is a combinatorial simplicial complex structure on $\Cone(\Link_{X}(\Sigma))$ and an extension
$$h:\Cone(\Link_{X}(\Sigma)) \m
Y(\boundcomp) * \sphere(M_1) * \cdots * \sphere(M_k).$$
Let $Z$ be the simplicial complex structure on $S^\thedegree$ obtained from $X$ by replacing $\Star(\Sigma)$ with
$$\partial \Sigma * \Cone(\Link_{X}(\Sigma))$$
and let $g:Z \m \sphere(M)$ be the map that is $f$ on the vertices of $Z$ that are also vertices of $X$ and is $h$ on the new vertices.

We claim that $g$ maps to $\geomPD$ and that $f$ and $g$ are homotopic as maps to $\geomPD$. We can do this by checking that if $\theta$ is a simplex of $$Y(\boundcomp) * \sphere(M_1) * \dots * \sphere(M_k),$$ then $\theta * f(\Sigma)$ is a simplex of $\geomPD$. This follows from \cref{inD} since $f(\Sigma)$ and $\theta * f(\Sigma)$ have the same pillar.

The proof of \cite[Theorem 3.5]{Aygun2022} implies that $Z$ has one less maximal-dimensional simplex that maps to a degree-$\themaxdegree$ pillar than $X$. Note that \cref{3.2} implies that if the image of a map to $\sphere$ contains a simplex of degree $\themaxdegree$, then it also contains a pillar of degree $\themaxdegree$. Thus, iterating this process produces a homotopy through $\geomPD$ to a map to $\sphere^{\leq n-2} \subseteq \geomPD$. As noted earlier, the complex $\sphere^{\leq n-2}$ is $(2n-4)$-connected so we conclude that $f$ is null-homotopic.
\end{proof}

We now prove \cref{thm:mainCB}, which says that $\CB \simeq \bigvee S^{2n-3}$.

\begin{proof}[Proof of \cref{thm:mainCB}]
    Since $\CB \simeq \PD$ (\cref{thm:CBandPD}) and $\PD$ is $(2n-3)$-dimensional, it suffices to show that $\CB$ is $(2n-4)$-connected. Since $\CB \simeq \geomPD$ (see \cref{eq:he_CB_PD_geomPD}), and $\geomPD$ is $(2n-4)$-connected (\cref{prop:connectivity_D}), the result follows. 
\end{proof}

\printbibliography

\end{document}